\documentclass[a4paper]{article}
\usepackage{geometry}
\usepackage{graphicx,url}
\usepackage[margin=10pt, font=small, labelfont=bf]{caption}

\usepackage[english]{babel}   
\usepackage[T1]{fontenc}
\usepackage[utf8]{inputenc}

\usepackage{slashed}
\usepackage{mathtools}

\usepackage{tikz}
\usepackage{verbatim}
\usepackage{listings}
\usepackage{xcolor}
\usepackage{amsmath,amssymb,mathrsfs}
\usepackage{mathtools}
\usepackage{amsthm}
\usepackage{frontespizio}
\usepackage{pictexwd, dcpic}
\usepackage{epigraph}
\usepackage{enumerate}
\usepackage{fancyhdr}
\usepackage{textcomp}
\usepackage{cancel}
\usepackage{bbm}
\usepackage{authblk}
\theoremstyle{plain}
\newtheorem{thm}{Theorem}[subsection]
\newtheorem{prop}[thm]{Proposition}
\newtheorem{cor}[thm]{Corollary}
\newtheorem{lem}[thm]{Lemma}
\theoremstyle{definition}
\newtheorem{defn}[thm]{Definition}
\newtheorem{rem} [thm] {Remark}

\title{Local and Global well-posedness in the $L^2$-setting for the $Q$-tensor model in $\mathbb R^N$ and $\mathbb R^N_+$}
\date{}
\author{Daniele Barbera, Vladimir Georgiev, Miho Murata, Yoshihiro Shibata}

\begin{document}
 \maketitle

\begin{abstract}
The paper studies the $Q$-tensor model for nematic liquid crystals, a system that couples a Navier-Stokes equation with an evolution equation for the order parameter tensor $Q$. The first goal of the paper is to establish the local well-posedness of the system in $\mathbb R^N$ and $\mathbb R^N_+$ for $N=2,3$ in the $L^2$ framework, improving existing results in the literature, where the existence of local strong solutions was obtained only under smallness assumptions on the initial data. Fundamental is the application of the energy method, which shows a cancellation phenomena on the nonlinear terms, allowing the use of a contraction argument to prove existence and uniqueness of solutions. Finally, with the same approach we establish global well-posedness in the three-dimensional case for small initial data.
\end{abstract}

{\textbf{MSC Numbers}: 35A01, 35Q35, 76A15}

{\textbf{Keywords}: Liquid Crystals, $Q$-tensor, Energy Method, Local well-posedness.}

\section{Introduction and mathematical background}

Let us consider the $Q$-tensor model (or Beris-Edwards model) on a domain $\Omega\subseteq\mathbb R^N$ with $N\ge 2$:
\begin{equation}\label{BE.sys.}
    \left\{\begin{array}{ll}
       (\partial_t-\Delta)u+\nabla \pi+\beta {\rm Div}(\Delta-a)Q= F(u,Q)  & (0,T)\times \Omega \\
       (\partial_t-\Delta+a)Q-\beta D(u)=G(u,Q)  & (0,T)\times \Omega \\
       {\rm div} u=0 & (0,T)\times \Omega \\
       u=0,\quad \partial_\nu Q=0 & (0,T)\times\partial\Omega \\
       u(0)=u_0,\quad Q(0)=Q_0 & \Omega,
    \end{array}\right.
\end{equation}
where $\nu(x)$ is the external vector of $\partial\Omega$ in $x\in\partial\Omega$ and where
$$ F(u,Q) =-(u\cdot \nabla) u + {\rm Div}\left[2\xi \mathbb{H}\colon Q\left(Q+\frac{Id}{N}\right)-(\xi+1)\mathbb{H}Q+(1-\xi)Q\mathbb{H}-\nabla Q\odot\nabla Q\right]-\beta {\rm Div}\mathcal{L}[\mathcal{F}(Q)],  $$
$$ G(u,Q)=-(u\cdot\nabla)Q+\xi(D(u)Q+QD(u))+W(u)Q-QW(u)-2\xi\left(Q+\frac{Id}{N}\right)Q\colon \nabla u+\mathcal{L}[\mathcal{F}(Q)],  $$
$$ D(u)=\frac{1}{2}\left(\nabla u + \nabla^Tu\right), \quad  W(u)=\frac{1}{2}\left(\nabla u-\nabla^Tu\right), $$
$$ [\nabla Q\odot\nabla Q]_{jk}=\sum_{\alpha,\beta=1}^N\partial_j Q_{\alpha\beta}\partial_k Q_{\alpha\beta}\quad j,k=1,\ldots, N, $$
$$ \mathbb{H}=\Delta Q-aQ+b\mathcal{L}[Q^2]-c|Q|^2Q, \quad \mathcal{F}(Q)=bQ^2-c|Q|^2Q,  $$
$$ \mathcal{L}[A]=A-{\rm tr}(A)\frac{Id}{N},\quad A\colon B={\rm tr}\left(B^TA\right), \quad {\rm Div}A=\sum_{i=1}^N \partial_i A^i \quad \forall A,B\in \mathbb{R}^{N\times N}, $$
with $a,T>0$, $\xi,b,c\in\mathbb{R}$, $\beta=\frac{2\xi}{N}$ and where $S_0(N,\mathbb R)$ is the set of symmetric traceless matrices, that is
\begin{equation}\label{def.S0}
    S_0(N,\mathbb R)=\left\{A\in\mathbb R^{N\times N}\:\Big|\:{\rm tr}A=0,\quad A=A^T\right\}.
\end{equation}

\vspace{2mm}

The model was introduced by Beris and Edwards in \cite{BE94} to describe the behaviour of nematic liquid crystals. Liquid crystals are an intermediate state of matter, between the solid state and the liquid state. Historically, due to its optical properties, such a material was employed in the production of LCDs and LCFs (see \cite{K02}, \cite{GC15}). However, nowadays the applications for liquid crystals are wider. One of the most interesting is in the pharmaceutical area: it has been observed that some particular nano-structure, called cellulose nano-crystals or CNCs, used as drug delivery material, shows a liquid crystals behaviour. For more details, see \cite{XAS20}, \cite{GK23} and \cite{SO23}. Nematic liquid crystals, in particular, are made of particles with no positional order but with an orientation in the space. From this perspective, looking at the system \eqref{BE.sys.}, the functions $u\colon(0,T)\times \Omega\to \mathbb{R}^N$, $\pi\colon(0,T)\times \Omega\to\mathbb{R}$ and $Q\colon(0,T)\times \Omega\to S_0(N,\mathbb{R})$ represent, respectively, the velocity field, the pressure and the orientation matrix, called $Q$-tensor, introduced by de Gennes in \cite{GP93}.

\vspace{2mm}

From a mathematical point of view, the $Q$-tensor model is formed by a Navier-Stokes equation for $u$ coupled with a parabolic equation for $Q$. Most of the works that study the existence of solutions for the $Q$-tensor model focuses on weak solutions and consider the additional constraint $\xi=0$. In fact, when $\xi=0$, it can be noticed that several nonlinear terms vanish and, most importantly, the corresponding linear system turns diagonal, formed by a Stokes equation for $u$ and a parabolic equation for $Q$. In this setting, we recall the results of \cite{PZ12} and \cite{DA17}, with the existence of a global weak solution in $\mathbb R^N$ with $N=2,3$, while in \cite{ADL16}, \cite{X17}, \cite{GR15} and \cite{LLZ19} it is considered the case of a bounded domain. For what concerns the case $\xi\in\mathbb R$, we cite the works \cite{ADL14} and \cite{LW18} for the bounded case and \cite{LWX19} for the periodic case in two dimensions.

\vspace{2mm}

In this paper, we focus on strong solutions without any additional constraint on $\xi$. The existence of strong solutions with $\Omega$ bounded was already investigated by \cite{HHW24}, while for the periodic case we refer to the results of \cite{CR16}. In our work, we consider $\Omega=\mathbb R^N,\mathbb R^N_+$ with $N=2,3$. The existence of strong solutions in $\mathbb R^N$ was studied in \cite{SS19} and \cite{MS22}. Here, the authors proved the $L^p-L^q$ maximal regularity and the global well-posedness of the $Q$-tensor model in $\mathbb R^N$ with $N\ge 3$. Later, the case of $\Omega=\mathbb R^N_+$ with $N\ge 2$ was studied in \cite{BM24}, with the $L^p-L^q$ maximal regularity, and in \cite{BMS25} with the global well-posedness in homogeneous Sobolev spaces.

\subsection{Notations and functional spaces}

As we mentioned above, we consider the case $\Omega=\mathbb R^N,\mathbb R^N_+$ with $N=2,3$. Here $\partial\Omega=\mathbb R^N_0$, where
$$ \mathbb R^N_0=\left\{(x,0)\in\mathbb R^N\:\Big|\:x\in\mathbb R^{N-1}\right\}. $$
For simplicity, we write
$$     \left\{\begin{array}{ll}
       (\partial_t-\Delta)u+\nabla \pi+\beta {\rm Div}(\Delta-a)Q= F(u,Q) & (0,T)\times \Omega \\
       (\partial_t-\Delta+a)Q-\beta D(u)=G(u,Q)  & (0,T)\times \Omega \\
       {\rm div} u=0 & (0,T)\times \Omega \\
       u=0,\quad \partial_\nu Q=0 & (0,T)\times\partial\Omega \\
       u(0)=u_0,\quad Q(0)=Q_0 & \Omega,
    \end{array}\right. $$
to consider both systems
$$ \left\{\begin{array}{ll}
       (\partial_t-\Delta)u+\nabla \pi+\beta {\rm Div}(\Delta-a)Q= F(u,Q)  & (0,T)\times \mathbb R^N \\
       (\partial_t-\Delta+a)Q-\beta D(u)=G(u,Q)  & (0,T)\times \mathbb R^N \\
       {\rm div} u=0 & (0,T)\times \mathbb R^N \\
       u(0)=u_0,\quad Q(0)=Q_0 & \mathbb R^N
    \end{array}\right. $$
    and 
$$ \left\{\begin{array}{ll}
       (\partial_t-\Delta)u+\nabla \pi+\beta {\rm Div}(\Delta-a)Q= F(u,Q)  & (0,T)\times \mathbb R^N_+ \\
       (\partial_t-\Delta+a)Q-\beta D(u)=G(u,Q)  & (0,T)\times \mathbb R^N_+ \\
       {\rm div} u=0 & (0,T)\times \mathbb R^N_+ \\
       u=0,\quad \partial_{N} Q=0 & (0,T)\times\mathbb R^N_0 \\
       u(0)=u_0,\quad Q(0)=Q_0 & \mathbb R^N_+.
    \end{array}\right. $$
We use the same notation for the corresponding linear systems. 

For what concerns the functional spaces, let $p,q\in(1,\infty)$, $m\in\mathbb{N}_0=\mathbb N\cup\{0\}$ and $s\in\mathbb R$, then we denote $L^q(\Omega;\mathbb R^N)$, $W^{m,q}(\Omega;\mathbb R^N)$ and $B^s_{q,p}(\Omega;\mathbb R^N)$ respectively the Lebesgue, the Sobolev and the Besov spaces with values in $\mathbb R^N$ and we denote $\|\cdot\|_{L^q(\Omega)}$, $\|\cdot\|_{W^{m,q}(\Omega)}$ and $\|\cdot\|_{B^s_{q,p}(\Omega)}$ their norms. When the function takes place in $\mathbb R$, we write for simplicity $L^q(\Omega)$, $W^{m,q}(\Omega)$ and $B^s_{p,q}(\Omega)$. We  denote $H^m(\Omega)\coloneqq W^{m,2}(\Omega)$ and 
$$ H^1_0(\Omega)=\left\{f\in H^1(\Omega)\:\Big|\:f_{|\partial\Omega}=0\right\}. $$
Let $s\in(0,1)$ and $p\in(1,\infty)$, then we recall the definition of
$$ H^s_p(\mathbb{R})\coloneqq \left\{v\in L^p(\mathbb{R})\:\Big|\: \mathcal{F}^{-1}[(1+|\tau|^2)^{s/2}\mathcal{F}[v]]\in L^p(\mathbb{R})\right\} $$
endowed with the norm
$$ \|v\|_{H^s_p(\mathbb{R})}\coloneqq \left\|\mathcal{F}^{-1}\left[(1+|\tau|^2)^{s/2}\mathcal
F [u]\right]\right\|_{L^p(\mathbb{R})}, $$
where $\mathcal{F}$ and $\mathcal{F}^{-1}$ are respectively the Fourier and the Inverse Fourier transforms, that is 
$$ \mathcal{F}[f](\tau)=\int_{\mathbb R}e^{-it\tau}f(t)dt, \quad \mathcal{F}^{-1}[f](t)=\frac{1}{2\pi}\int_{\mathbb R}e^{it\tau}f(\tau)d\tau. $$
Let now $I\subseteq\mathbb{R}$ open, then we define
$$ H^s_p(I)\coloneqq\left\{v\in L^p(I)\mid \exists \:\widetilde{v}\in H^s_p(\mathbb{R})\:\:\text{such that}\:\: \widetilde{v}_{|I}=v\right\}, $$
with the norm
$$ \|v\|_{H^s_p(I)}\coloneqq \inf_{\widetilde{v}_{|I}=v}\|\widetilde{v}\|_{H^s_p(\mathbb{R})}. $$
Let $q\in(1,\infty)$, then we denote
$$ J_q\left(\Omega\right)\coloneqq \left\{f\in L^q\left(\Omega;\mathbb{R}^N\right)\:\Big|\:\left<f,\nabla \varphi\right>=0,\quad \forall \varphi\in\widehat{H}^1_{q^\prime}\left(\Omega\right)\right\}, $$
where
$$ \widehat{H}^1_{q}\left(\Omega\right)\coloneqq \left\{\varphi\in L^{q}_{loc}\left(\Omega\right)\:\Big|\: \nabla \varphi\in L^{q}\left(\Omega;\mathbb{R}^N\right)\right\}. $$
In particular, for some $\Omega\subseteq\mathbb R^N$ (see Theorem 3 at p. 129 of \cite{GN18}), $L^q(\Omega;\mathbb R^N)$ can be decomposed as
$$ L^q(\Omega;\mathbb R^N)=J_q(\Omega)\oplus G_q(\Omega), $$
where
$$ G_q(\Omega)=\left\{\nabla \varphi\:\Big|\:  \varphi \in \widehat H^1_q(\Omega)\right\}, $$
with the projection $\mathbb P_q\colon L^q(\Omega;\mathbb R^N)\to J_q(\Omega)$ called Helmholtz Projection. For $q=2$, $\mathbb P=\mathbb P_2$, the decomposition always exists and it is orthogonal (see Remark 1 at p. 128 of \cite{GN18}). Since, in our case, $\mathbb P_pf=\mathbb P_qf$ for any $f\in L^p(\Omega)\cap L^q(\Omega)$ (Proposition 3 at p. 130 of \cite{GN18}), we simply write $\mathbb P$ without specifying the parameter $q$. For the initial conditions of \eqref{BE.sys.}, we use the following spaces:
\begin{defn}\label{def.HA}\hfill\\
Let $A\colon D(A)\subseteq L^2(\Omega)\to L^2(\Omega)$ be a negative, self-adjoint operator, then we denote with $H^s_A(\Omega)$ the space $D((1-A)^{s/2})$ endowed with the norm
$$ \|f\|_{H^s_A(\Omega)}\coloneqq \left\|(1-A)^{s/2}f\right\|_{L^2(\Omega)}. $$
\end{defn}
In the following, we use the spaces $H^s_A(\Omega)$ with $A=\mathbb P\Delta_D,\Delta_D,\Delta_N$, that are respectively the Stokes operator with Dirichlet boundary conditions and the Laplacian operator with Dirichlet and Neumann boundary conditions. In particular
$$ H^1_{\mathbb P\Delta_D}\left(\Omega;\mathbb R^N\right)=H^1_0\left(\Omega;\mathbb R^N\right)\cap J_2(\Omega), $$
$$ H^1_{\Delta_D}\left(\Omega;\mathbb R^N\right)=H^1_0\left(\Omega;\mathbb R^N\right),\quad H^2_{\Delta_D}(\Omega;\mathbb R^N)=H^2(\Omega;\mathbb R^N)\cap H^1_0(\Omega;\mathbb R^N). $$
$$ H^2_{\Delta_N}\left(\Omega;\mathbb R^N\right)=\left\{f\in H^2\left(\Omega;\mathbb R^N\right)\mid \partial_\nu f_{|\partial\Omega}=0\right\}. $$
Moreover, let $X$ be a Banach space, then we denote $L^p((a,b);X)$, $W^{m,p}((a,b);X)$ and $H^s((a,b);X)$ the previous function spaces for $X$-valued functions for any $(a,b)\subseteq\mathbb{R}$. In particular, for simplicity we often write
$$ L^p_TX\coloneqq L^p((0,T);X), $$
for $p,q\in[1,\infty]$ and $T>0$. When $T=\infty$, we use the notation $L^pX$. 

For the derivatives, we use the following notations: for any multi-index $\alpha\in\mathbb{N}^N_0$ we write
$$ |\alpha|=\alpha_1+\cdots+\alpha_N, $$
$$ D^\alpha= \partial^{\alpha_1}_{x_1}\cdots \partial^{\alpha_N}_{x_N}. $$
For any $k\in\mathbb{N}_0$, for any $\Omega\subseteq\mathbb{R}^N$ open set and for any function $f\colon \Omega\to\mathbb{R}^N$ we denote
$$ \nabla^kf=(D^\alpha f\mid |\alpha|=k). $$
We also denote $\nabla^\prime=(\partial_1,\ldots,\partial_{N-1})$ and $\Delta^\prime=\nabla^\prime\cdot\nabla^\prime$. Finally, in the paper we  use $C$ to indicate a constant which depends on the parameters of the problem. In the statements, we use $C(a,b,\ldots)$ to underline the dependence from $a,b,\ldots$, otherwise we use the symbols 
$$ f(x)\lesssim g(x)\:\Leftrightarrow\: \exists C\:\: \text{s.t.}\:\:f(x)\le Cg(x) $$
$$ f(x)\gtrsim g(x)\:\Leftrightarrow\: \exists C\:\: \text{s.t.}\:\:f(x)\ge Cg(x) $$
$$ f(x)\sim g(x)\:\Leftrightarrow\: \exists C\:\: \text{s.t.}\:\:f(x)= Cg(x). $$

\subsection{Main results}

The aim of the paper is to prove the local and global well-posedness for the $Q$-tensor model \eqref{BE.sys.} in $\Omega=\mathbb R^N,\mathbb R^N_+$ with $N=2,3$. In this setting, the works we cited before (\cite{SS19},\cite{MS22},\cite{BM24},\cite{BMS25}) do not consider the local well-posedness (in \cite{SS19} and \cite{BM24} the existence of a local solution has been proved, but requiring the smallness of the initial conditions). Generally, once the $L^p-L^q$ maximal regularity is achieved, the local well-posedness in $L^p-L^q$ can be obtained with a standard contraction argument. However, the complexity of the nonlinear terms in \eqref{BE.sys.} makes this application difficult to employ. More precisely, some of the nonlinear terms involve high derivatives on the functions (we explain this point more in details in Remark \ref{rem.local-probl.}). To overcome this issue, we adopt the energy method: as we will see in the following, using the energy method in our problem causes some cancellations in the highest order terms and allows us to prove the contraction argument in the $L^2$-setting. For more details on the energy method technique, see Section 3, in particular the proof of Proposition \ref{p.res-unif.es.}. We introduce the spaces
\begin{equation}\label{def.X}
    X_T^s(\Omega)=L^\infty\left((0,T);H^s\left(\Omega; \mathbb R^N\right)\right)\cap L^2\left((0,T);H^{s+1}\left(\Omega;\mathbb R^N\right)\right)
\end{equation}
for $s\ge 0$ and
\begin{equation}\label{def.Y}
    Y_T(\Omega)= \left(X^1_T(\Omega)\times X^2_T(\Omega)\right)\cap \left(H^1\left((0,T);L^2\left(\Omega; \mathbb R^N\right)\right)\times H^1\left((0,T);H^1\left(\Omega; \mathbb R^N\right)\right)\right).
\end{equation}
We are finally ready to state the local well-posedness result:
\begin{thm}\label{t.loc.ex.}\hfill\\
Let $\Omega=\mathbb R^N,\mathbb R^N_+$ with $N=2,3$, let $a>0$ and $\beta\in\mathbb R$, let $u_0\in H^{1}_{\mathbb P\Delta_D}(\Omega;\mathbb R^N)$ and $Q_0\in H^2_{\Delta_N}(\Omega;S_0(N,\mathbb R))$, where $S_0(N,\mathbb R)$ and $H^s_A(\Omega)$ for $s\ge0$ are defined respectively in \eqref{def.S0} and in Definition \ref{def.HA}, then there is $T>0$ sufficiently small such that the $Q$-tensor system \eqref{BE.sys.} admits a solution $(u,\pi,Q)$, unique up to additive functions $c(t)$ on the pressure term, with 
$$ (u,Q)\in Y_T(\Omega),\quad \nabla \pi\in L^2\left((0,T);L^2\left(\Omega;\mathbb R^N\right)\right), $$
such that $u(t)\in H^2_{\mathbb P\Delta_D}(\Omega;\mathbb R^N)$, $Q(t)\in H^3_{\Delta_N}(\Omega;S_0(N,\mathbb R))$ and $\pi(t)\in L^2_{loc}(\Omega)$ for a.e. $t\in(0,T)$, where $Y_T(\Omega)$ is defined in \eqref{def.Y}.      
\end{thm}
As a corollary, we can recover an $L^p-L^q$ local existence result:
\begin{cor}\label{c.loc.ex.pq.}
Let $\Omega=\mathbb R^N,\mathbb R^N_+$ with $N=2,3$, let $a>0$ and $\beta\in\mathbb R$, let $p,q\in (1,2]$, with $\frac{2}{p}+\frac{1}{q}<2$ when $\Omega=\mathbb R^N_+$, let $u_0$ and $Q_0$ as in Theorem \ref{t.loc.ex.} with the additional condition
$$ u_0\in B^{2(1-1/p)}_{q,p}\left(\Omega;\mathbb R^N\right)\cap J_q(\Omega), \quad Q_0\in B^{3-2/p}_{q,p}\left(\Omega;S_0(N,\mathbb R)\right), $$
let $(u,\pi,Q)$ be the solution from Theorem \ref{t.loc.ex.}, then
$$ u\in \bigcap_{\ell=0}^2 H^{\ell/2}_p\left((0,T);W^{2-\ell,q}\left(\Omega;\mathbb R^N\right)\right), \quad Q\in \bigcap_{\ell=0}^2 H^{\ell/2}_p\left((0,T);W^{3-\ell,q}\left(\Omega;\mathbb R^N\right)\right), $$
$$ \nabla \pi\in L^p\left((0,T);L^q\left(\Omega;\mathbb R^N\right)\right). $$
\end{cor}
The estimate we are going to achieve by the energy method can be adopted to obtain the global well-posedness of the problem in three dimensions:
\begin{thm}\label{t.gl.ex.2}
Let $\Omega=\mathbb R^3,\mathbb R^3_+$, let $a>0$ and $\beta\in\mathbb R$, let $q\in\left(1,\frac{6}{5}\right)$
$$ u_0\in H^1_{\mathbb P\Delta_D}\left(\Omega;\mathbb R^3\right)\cap L^q\left(\Omega;\mathbb R^3\right), \quad Q_0\in H^2_{\Delta_N}\left(\Omega;S_0(3,\mathbb R)\right)\cap W^{1,q}\left(\Omega;S_0(3,\mathbb R)\right), $$
where $S_0(3,\mathbb R)$ and $H^s_A(\Omega)$ for $s\ge 0$ are defined respectively in \eqref{def.S0} and in Definition \ref{def.HA}, then there is $\varepsilon>0$ sufficiently small such that, when
$$ \|u_0\|_{H^1(\Omega)\cap L^q(\Omega)} + \|Q_0\|_{H^2(\Omega)\cap W^{1,q}(\Omega)}\le \varepsilon, $$
there is a solution $(u,\pi,Q)$ for the system \eqref{BE.sys.}, unique up to additive functions $c(t)$ on the pressure term, such that 
$$ (u,Q)\in Y(\Omega), \quad \nabla \pi\in L^2\left(\mathbb R_+;L^2\left(\Omega;\mathbb R^3\right)\right), $$
with $u(t)\in H^2_{\mathbb P\Delta_D}(\Omega;\mathbb R^N)$, $Q(t)\in H^3_{\Delta_N}(\Omega;S_0(N,\mathbb R))$ and $\pi(t)\in L^2_{loc}(\Omega)$ for a.e. $t>0$, where we recall the definition of $Y(\Omega)$ from \eqref{def.Y}.
\end{thm}
The difficulty behind the case $N=2$ comes from the $L^2(\mathbb R_+;L^2(\Omega;\mathbb R^N))$-norm of the function $u$. This phenomenon also appears in the heat case: let
$$ v(t,x)=e^{\Delta t}u_0(x)\quad (t,x)\in \mathbb R_+\times \mathbb R^N. $$
It is well-known the semigroup decay
$$ \|v(t)\|_{L^2(\mathbb R^N)}\le Ct^{-\frac{N}{2}\left(\frac{1}{q}-\frac{1}{2}\right)}\|u_0\|_{L^q(\Omega)} \quad q\in[1,2]. $$
In particular, when $N=2$, the decay rate is polynomial of order $\frac{1}{q}-\frac{1}{2}$. So, even assuming $u_0\in L^1(\Omega;\mathbb R^2)$, which gives the best decay rate, the function does not belong to $L^2(\mathbb R_+;L^2(\mathbb R^2;\mathbb R^2))$. To conclude, we observe that Theorem \ref{t.gl.ex.2} covers the $L^2$ setting case, which was not included in the results of \cite{BMS25}.

\vspace{2mm}

The paper is divided as follows: in Section \ref{sec.prel.} we prove some preliminary results that will be applied in sections \ref{sec.local} and  \ref{sec.global}, respectively devoted to the proof of the local and global well-posedness results, that is Theorems \ref{t.loc.ex.} and \ref{t.gl.ex.2}.

\bigskip

\thanks{
\textbf{Acknowledgements.} D.B. and V.G. are partially supported by INDAM, GNAMPA group, and by the "INdAM - GNAMPA Project", cod. CUP E53C25002010001. Y.S. is partially supported by JSPS KAKENHI, Grant Number JP23K22405. V.G. is partially supported  by Institute of Mathematics and Informatics, Bulgarian Academy of Sciences, by Top Global University Project, Waseda University. M.M. is partially supported by JSPS KAKENHI, Grant Numbers JP26K06877, JP23K22405}

\section{Preliminary results}\label{sec.prel.}
\subsection{The $Q$-tensor semigroup}

Let us consider the $Q$-tensor linear system:
\begin{equation}\label{BE.lin.sys.}
     \left\{\begin{array}{ll}
       (\partial_t-\Delta)u+\nabla \pi+\beta {\rm Div}(\Delta-a)Q= F & (0,T)\times \Omega \\
       (\partial_t-\Delta+a)Q-\beta D(u)=G  & (0,T)\times \Omega \\
       {\rm div} u=0 & (0,T)\times \Omega \\
       u=0,\quad \partial_\nu Q=0 & (0,T)\times\partial\Omega \\
       u(0)=u_0,\quad Q(0)=Q_0 & \Omega,
    \end{array}\right.
\end{equation}
The $L^p-L^q$ maximal regularity result for \eqref{BE.lin.sys.} is proved in \cite{MS22} and \cite{BM24} :
\begin{thm}\label{t.max.reg.}[Theorem 2.1 of \cite{MS22} and Theorem 1.2.2 of \cite{BM24}]\hfill\\
Let $N\ge 2$ and $\Omega=\mathbb R^N,\mathbb R^N_+$, let $a>0$, $\beta\in\mathbb{R}$ and $p,q\in(1,+\infty)$, with $\frac{2}{p}+\frac{1}{q}<2$ when $\Omega=\mathbb R^N_+$, let $T_0>0$ and $T\in(0,T_0)$, let 
$$ F\in L^p\left((0,T);L^q\left(\Omega;\mathbb{R}^N\right)\right),\quad  G\in L^p\left((0,T);W^{1,q}\left(\Omega;S_0(N,\mathbb R)\right)\right), $$
$$ u_0\in B^{2(1-1/p)}_{q,p}\left(\Omega;\mathbb{R}^N\right)\cap J_q(\Omega), \quad Q_0\in B^{3-2/p}_{q,p}\left(\Omega;S_0(N,\mathbb R)\right), $$
where $S_0(N,\mathbb R)$ is defined in \eqref{def.S0}, then we can find a solution $(u,\pi,Q)$ for the linear system \eqref{BE.lin.sys.} in $(0,T)$, unique up to additive functions $c(t)$ on the pressure term, with $\pi(t)\in L^q_{loc}(\Omega)$ for a.e. $t\in(0,T)$ and
$$ u\in \bigcap_{\ell=0}^2H^{\ell/2}_p\left((0,T);W^{2-\ell,q}\left(\Omega;\mathbb{R}^{N}\right)\right),\:\: \nabla \pi\in L^p\left((0,T);L^q\left(\Omega;\mathbb{R}^N\right)\right), $$
$$ Q\in \bigcap_{\ell=0}^2H^{\ell/2}_p\left(\mathbb{R}_+;W^{3-\ell,q}\left(\Omega;S_0(N,\mathbb{R})\right)\right), $$
such that
$$ \sum_{\ell=0}^2\|(u,Q)\|_{H^{\ell/2}_p((0,T);W^{2-\ell,q}(\Omega)\times W^{3-\ell,q}(\Omega))} + \|\nabla \pi\|_{L^p((0,T);L^q(\Omega))} $$
$$ \le C(a,\beta,p,q,T_0)\left[\|f\|_{L^p((0,T);L^q(\Omega))}+\|g\|_{L^p((0,T);W^{1,q}(\Omega))}+\|u_0\|_{B^{2(1-1/p)}_{q,p}(\Omega)}+\|Q_0\|_{B^{3-2/p}_{q,p}(\Omega)}\right]. $$
\end{thm}
To be underlined that the authors in \cite{MS22} state the result for $\Omega=\mathbb R^N$ only for $N\ge 3$. However, it can be seen that the linear estimate of the cited paper holds also for $N=2$.
\begin{rem}\label{rem.local-probl.}
From the a priori estimate of Theorem \ref{t.max.reg.}, we can understand the complexity of applying a direct contraction argument for the local existence: looking at the definition of the nonlinear terms $F(u,Q)$ and $G(u,Q)$, to apply Theorem \ref{t.max.reg.} we need to control the quantity
$$ \|\nabla^3 QQ\|_{L^p_TL^q} + \|\nabla^2 u Q\|_{L^p_TL^q}. $$
Asking $p>2$ and $q>N$, it can be seen that $\|Q\|_{L^\infty_TL^\infty}<+\infty$. Therefore
\begin{equation}\label{aux.local-probl.}
    \|\nabla^3 QQ\|_{L^p_TL^q} + \|\nabla^2 u Q\|_{L^p_TL^q}\le \|Q\|_{L^\infty_TL^\infty}\left[\|\nabla^3 Q\|_{L^p_TL^q} + \|\nabla^2 u\|_{L^p_TL^q}\right]<+\infty.
\end{equation}
However, to apply the contraction argument it is necessary an estimate of the type 
$$ \|\nabla^3 QQ\|_{L^p_TL^q} + \|\nabla^2 u Q\|_{L^p_TL^q}\le C(T)\|(u,Q)\|_{Y_T(\Omega)}^2 $$
with $C(T)\to 0$ as $T\to0^+$ and the inequality \eqref{aux.local-probl.} seems difficult to improve in our setting.
\end{rem}
It is possible to prove the existence of the semigroup corresponding to the linear system \eqref{BE.lin.sys.}. We explain briefly the argument, for more details see Section 4.1 of \cite{BM24}: we need to express $\pi$ in terms of $u$ and $Q$. Using Lemma 2.2.1 at p.73 and Lemma 1.4.2 at p.202 of \cite{So01}, it can be seen that there are $K_O(u,Q),K_e(F)\in L^p((0,T);L^q_{loc}(\Omega))$ such that
$$ \pi= K_O(u,Q) + K_e(F). $$
Let us define then the operator $B=(B_1,B_2)$, with
\begin{equation}\label{def.op.B}
    B_j(u,Q)=\left\{\begin{array}{ll}
        \Delta u -\nabla K_O(u,Q) - \beta {\rm Div}(\Delta-a)Q & j=1 \\
        (\Delta-a)Q + \beta D(u) & j=2.
    \end{array}\right.
\end{equation}
Finally, in \cite{MS22} and \cite{BM24} it is proved the existence of a semigroup corresponding to the operator $B$ in a proper domain $D(B)$:
\begin{thm}\label{t.sem.es.}\hfill\\
Let $\Omega=\mathbb R^N,\mathbb R^N_+$ with $N\ge 2$, let $a>0$ and $\beta\in\mathbb{R}$, let 
$$ B\colon D(B)\to L^2\left(\Omega;\mathbb R^N\right)\times H^1(\Omega;S_0(N,\mathbb R)) $$
be the operator defined in \eqref{def.op.B}, where $S_0(N,\mathbb R)$ comes from \eqref{def.S0} and
$$ D(B)= D_1(B)\times D_2(B), $$
where
$$ D_1(B)=\left\{u\in H^2\left(\Omega;\mathbb R^N\right)\cap J_2(\Omega)\:\Big|\: u=0\quad \text{on}\:\:\mathbb \partial\Omega\right\} $$
$$ D_2(B)= \left\{Q\in H^3(\Omega;S_0(N,\mathbb R))\: \Big|\: \partial_\nu Q=0 \quad \text{on}\:\:\partial\Omega\right\}, $$
then $B$ generates a $C_0$-analytic semigroup $\{e^{Bt}\}_{t>0}$ with base space
$$ J_2(\Omega)\times H^1(\Omega;S_0(N,\mathbb R)). $$
Moreover,
$$ \left\|e^{Bt}(u_0,Q_0)\right\|_{L^2(\Omega)\times H^1(\Omega)}\le C(a,\beta,N)\left[ \|u_0\|_{L^2(\Omega)} + \|Q_0\|_{H^1(\Omega)}\right] \quad t>0 $$
and there is $\gamma_0>0$ such that, for any $\gamma>\gamma_0$
$$ \left\|e^{-\gamma t}e^{Bt}(u_0,Q_0)\right\|_{H^1(\mathbb R_+;L^2(\Omega)\times H^1(\Omega))} + \left\|e^{-\gamma t}e^{Bt}(u_0,Q_0)\right\|_{L^2(\mathbb R_+,H^2(\Omega)\times H^3(\Omega))} $$
$$ \le C(a,\beta,N)\left[\|u_0\|_{H^1(\Omega)} + \|Q_0\|_{H^2(\Omega)}\right]. $$
\end{thm}
The last inequality comes from the maximal regularity estimate of Theorem \ref{t.max.reg.} applied to \eqref{BE.lin.sys.} with $(F,G)=(0,0)$.

\subsection{Yosida operators properties}

Let $\varepsilon>0$. We recall the definition of the Yosida operators
\begin{equation}\label{Yoshida-res.}
     R_\varepsilon^Su=\left(1-\varepsilon\mathbb P\Delta_D\right)^{-1}u, \quad R_\varepsilon^HQ=\left(1-\varepsilon\Delta_N\right)^{-1}Q,
\end{equation}
that correspond respectively to the Stokes operator with Dirichlet boundary conditions and the Laplacian operator with Neumann boundary conditions. Before stating the properties of $R_\varepsilon^S$ and $R_\varepsilon^H$, we recall the following result, which is a consequence of Theorems IV.2.1 and IV.3.2 of \cite{G11}:
\begin{thm}\label{t.res.reg.}
Let $\Omega=\mathbb R^N,\mathbb R^N_+$ with $N\ge 2$, let $f\in H^m(\Omega;\mathbb R^N)$ with $m\in\mathbb N$, then we can find $v\in H^{m+2}_{loc}(\Omega;\mathbb R^N)$ and $\pi\in H^{m+1}_{loc}(\Omega)$ that solve the system 
$$ \left\{\begin{array}{ll}
    \Delta v=\nabla \pi + f & \Omega \\
    {\rm div}v=0 & \Omega \\
    v=0 & \partial\Omega,
\end{array}\right. $$
and for any $\ell\in\{0,\ldots, m\}$ it holds
$$ \left\|\nabla^{\ell+2}v\right\|_{L^2(\Omega)} + \left\|\nabla^{\ell+1}\pi\right\|_{L^2(\Omega)}\le C(N,m) \left\|\nabla^\ell f\right\|_{L^2(\Omega)}. $$
Moreover, if $(v_1,\pi_1)$ is another solution of the system with $\|\nabla^{\ell+2}v_1\|_{L^2(\Omega)}<+\infty$ for some $\ell\in\{0,\ldots, m\}$, then
$$ \left\|\nabla^{\ell+2}(v_1-v)\right\|_{L^2(\Omega)}=\left\|\nabla^{\ell+1}(\pi_1-\pi)\right\|_{L^2(\Omega)}=0. $$
\end{thm}
Thanks to this result, we can verify some important properties for the Yosida approximation operators:
\begin{lem}\label{l.res.prop.}
Let $u\in H^2_{\mathbb P\Delta_D}(\Omega)$ and $Q\in H^3_{\Delta_N}(\Omega)$, let $\varepsilon>0$ and let $R_\varepsilon^S$ and $R_\varepsilon^H$ from \eqref{Yoshida-res.}. Then, the following identities hold:
\begin{equation}\label{R_eps_ident}
R_\varepsilon^S u \;=\; u + \varepsilon\,\Delta R_\varepsilon^S u, 
\qquad 
R_\varepsilon^H Q \;=\; Q + \varepsilon\,\Delta R_\varepsilon^H Q.
\end{equation}
Moreover
$$ \Delta R_\varepsilon^S u \;\in\; H^2_{\mathbb{P}\Delta_D}(\Omega),
\qquad
\Delta R_\varepsilon^H Q \;\in\; H^3_{\Delta_N}(\Omega). $$
\end{lem}
\begin{proof}\hfill\\
\textbf{Step 1: Identities \eqref{R_eps_ident}.}
By definition of the Yosida operators,
$$ (1-\varepsilon \mathbb{P}\Delta_D)\, R_\varepsilon^S u = u,
\qquad
(1-\varepsilon \Delta_N)\, R_\varepsilon^H Q = Q. $$
Hence
$$ u = R_\varepsilon^S u - \varepsilon\,\mathbb{P}\Delta_D R_\varepsilon^S u,
\qquad
Q = R_\varepsilon^H Q - \varepsilon\,\Delta R_\varepsilon^H Q. $$
Since $R_\varepsilon^S u \in H^2_{\mathbb{P}\Delta_D}(\Omega)$, Theorem~\ref{t.res.reg.},
applied with $m=2$, ensures $R_\varepsilon^S u\in H^4(\Omega)$ and therefore
$\operatorname{div} \Delta R_\varepsilon^S u=0$. In particular, 
$$ \mathbb{P}\Delta R_\varepsilon^S u = \Delta R_\varepsilon^S u. $$

\bigskip
\noindent \textbf{Step 2: Regularity of $\Delta R_\varepsilon^S u$ and $\Delta R_\varepsilon^H Q$.} Firstly, since $Q\in H^3_{\Delta_N}(\Omega)$ it is clear that $R_\varepsilon^HQ\in H^3_{\Delta_N}(\Omega)$. So, the second identity of \eqref{R_eps_ident} yields 
$$ \Delta R_\varepsilon^HQ\in H^3_{\Delta_N}(\Omega). $$
On the other hand, in Step 1 we have already shown that $R_\varepsilon^S u\in H^4(\Omega)$ and $\mathbb P\Delta R_\varepsilon^Su=\Delta R_\varepsilon^Su$. Finally, from the first identity of \eqref{R_eps_ident} we get
$$ \Delta R_\varepsilon^S u\in H^2(\Omega)\cap H^1_0(\Omega)\cap J_2(\Omega)
=H^2_{\mathbb{P}\Delta_D}(\Omega). $$
\end{proof}
\begin{rem}\label{rem.trace-R_eps.}
It can be noticed that $\partial_j^2R_\varepsilon^Su\in H^2_{\mathbb P\Delta_D}(\Omega)$ for $j=1,\ldots, N$. In fact, when $\partial_j\in\nabla^\prime$, since $R_\varepsilon^S u\in H^4(\Omega)$ it holds 
$$ {\rm div}\partial_j^2R_\varepsilon^S u=\partial_j^2 {\rm div}R_\varepsilon^S u=0. $$
Moreover, it is clear that $\partial_j^2R_\varepsilon^Su=0$ on $\mathbb R^N_0$ when $\partial_j\in\nabla^\prime$. Finally, 
$$ \partial_N^2R_\varepsilon^Su=(\Delta-\Delta^\prime)R_\varepsilon^Su\in H^2_{\mathbb P\Delta_D}(\Omega), $$
where we recall that $\Delta^\prime=\nabla^\prime\cdot\nabla^\prime$. 
\end{rem}
We state the following result:
\begin{lem}\label{l.res.eq.}
Let $\Omega=\mathbb R^N,\mathbb R^N_+$ with $N=2,3$, let $A=\mathbb P\Delta_D,\Delta_D,\Delta_N$, let $s\in[0,2]$, then for any $f\in H^{s}_A(\Omega)$ it holds
  $$ \|(1-A)^{s/2}f\|_{L^2(\Omega)} \sim \|f\|_{H^s(\Omega)}. $$
\end{lem}
The case $A=\mathbb P\Delta_D,\Delta_D,\Delta_N$ is proved in Lemma 2.2 of \cite{BG25} for $N\ge 3$. However, the proof is based on the resolvent estimates for the Stokes and Laplacian operators, so it can be extended to the case $N=2$. 

We need a similar result for the operator $B$, defined in \eqref{def.op.B}, but it can be seen that $B$ is not self-adjoint. Nevertheless, we know from Theorem \ref{t.sem.es.} that $B$ generates an analytic semigroup and this is enough to define the fractional operator $(\lambda-B)^\alpha$ for $\alpha,\lambda>0$ (see Section 2.6 of \cite{P83}). So, we can define as in Definition \ref{def.HA} the space $H^s_B(\Omega)$ for $s\ge 0$. Moreover, the proof of Lemma \ref{l.res.eq.} is based on the interpolation and the resolvent estimates, which hold true for $B$ as a consequence of Theorem 2.3 of \cite{MS22} and Theorem 1.2.3 of \cite{BM24}. Therefore, with the same strategy of Lemma 2.2 of \cite{BG25}, it can be proved the result also for the operator $B$:
\begin{lem}\label{l.res.eq.B}
    Let $\Omega=\mathbb R^N,\mathbb R^N_+$ with $N=2,3$, let $s\in[0,2]$, then for any $(f,g)\in H^{s}_B(\Omega)$ it holds
    $$ \|(1-B)^{s/2}(f,g)\|_{L^2(\Omega)\times H^1(\Omega)}\sim \|(f,g)\|_{H^s(\Omega)\times H^{s+1}(\Omega)}, $$
where the operator $B$ is defined in \eqref{def.op.B}.
\end{lem}
Next, we want to emphasize the role of $\varepsilon$ in the estimates for the Yosida operators: 
\begin{lem}\label{l.res.es.}
Let $\Omega=\mathbb R^N,\mathbb R^N_+$ with $N=2,3$, let $\varepsilon\in(0,1)$, $j,k=0,1,2$ with $j\le k$, let $A=\mathbb P\Delta_D,\Delta_D,\Delta_N$ and $w\in H^j_A(\Omega)$, then it holds
$$ \left\|(1-\varepsilon A)^{-1}w\right\|_{H^k(\Omega)}\lesssim \varepsilon^{-\frac{k-j}{2}}\|w\|_{H^{j}(\Omega)}. $$
Moreover, if $w\in H^{j+1}_{\Delta_N}(\Omega)$, it holds
$$ \left\|\nabla (1-\varepsilon\Delta_N)^{-1}w\right\|_{H^k(\Omega)}\lesssim \varepsilon^{-\frac{k-j}{2}}\|\nabla w\|_{H^{j}(\Omega)}. $$
\end{lem}
\begin{proof}\hfill\\
Using the resolvent estimates for the related operators, we know that
\begin{equation}\label{proof.res.es.}
    \|\nabla^k(1-\varepsilon A)^{-1}w\|_{L^2(\Omega)}\lesssim \varepsilon^{-k/2}\|w\|_{L^2(\Omega)},
\end{equation}
for $k=0,1,2$. By Lemma \ref{l.res.eq.}
$$ \|(1-\varepsilon A)^{-1}w\|_{H^k(\Omega)} \lesssim \|(1-A)^{k/2}(1-\varepsilon A)^{-1}w\|_{L^2(\Omega)} $$
$$ = \|(1-A)^{\frac{k-j}{2}}(1-\varepsilon A)^{-1}(1-A)^{j/2}w\|_{L^2(\Omega)}\lesssim \|(1-\varepsilon A)^{-1}(1-A)^{j/2}w\|_{H^{k-j}(\Omega)}, $$
so we conclude applying \eqref{proof.res.es.} and Lemma \ref{l.res.eq.}. For what concerns the second estimate with $A=\Delta_N$ and $w\in H^{j+1}_{\Delta_N}(\Omega)$, the proof is clear when $\Omega=\mathbb R^N$, so we focus on the half-space case. It can be seen that
$$ \partial_i(1-\varepsilon\Delta_N)^{-1}w=(1-\varepsilon\Delta_N)^{-1}\partial_iw \quad i<N, $$ 
$$ \partial_N(1-\varepsilon \Delta_N)^{-1}w=(1-\varepsilon\Delta_D)^{-1}\partial_Nw. $$
So we conclude applying the inequality for $A=\Delta_D,\Delta_N$ and $j=0,1,2$.
\end{proof}
\begin{lem}\label{l.res.limit}
Let $\Omega=\mathbb R^N,\mathbb R^N_+$ with $N=2,3$ and $A=\mathbb P\Delta_D,\Delta_D,\Delta_N$, let $w\in H^2_A(\Omega)$, 
$$ (1-\varepsilon A)^{-1}w \;\longrightarrow\; w
\qquad\text{in } H^2(\Omega)
\quad\text{as }\varepsilon\to0^+. $$
Moreover, if $w\in H^3_{\Delta_N}(\Omega)$, then
$$ (1-\varepsilon\Delta_N)^{-1}w \;\longrightarrow\; w
\qquad\text{in }H^3(\Omega)
\quad\text{as }\varepsilon\to0^+. $$
\end{lem}
\begin{proof}\hfill\\
\textbf{Step 1: Convergence in $H^2$.} Applying Lemma \ref{l.res.eq.}, we obtain
$$ \|(1-\varepsilon A)^{-1}w-w\|_{H^2(\Omega)}
\;\lesssim\;
\|(1-A)((1-\varepsilon A)^{-1}w-w)\|_{L^2(\Omega)}. $$
Since $(1-\varepsilon A)^{-1}$ commutes with $(1-A)$ on $D(A)$, it holds
$$ (1-A)(1-\varepsilon A)^{-1}w-(1-A)w
=  (1-\varepsilon A)^{-1}(1-A)w-(1-A)w. $$
Moreover, $(1-\varepsilon A)^{-1}\to Id$ strongly in $L^2(\Omega)$ as $\varepsilon\to0^+$
and $(1-A)w\in L^2(\Omega)$, thus
$$ (1-\varepsilon A)^{-1}w \to w \quad\text{in }H^2(\Omega). $$

\bigskip
\noindent \textbf{Step 2: Convergence in $H^3$ for the Neumann Laplacian.}
Let $w\in H^3_{\Delta_N}(\Omega)$.
We write
$$ \|(1-\varepsilon\Delta_N)^{-1}w-w\|_{H^3(\Omega)}
\le 
\|(1-\varepsilon\Delta_N)^{-1}w-w\|_{L^2(\Omega)}
+
\|\nabla (1-\varepsilon\Delta_N)^{-1}w - \nabla w\|_{H^2(\Omega)}. $$
The first term converges to $0$ by Step~1.
Thus, it suffices to treat the gradient term.

\smallskip
\emph{Tangential derivatives.}
For $i<N$,
tangential derivatives commute with $\Delta_N$ and preserve Neumann
boundary conditions:
$$ \partial_i \Delta_N = \Delta_N \partial_i,
\qquad
\partial_i w \in H^2_{\Delta_N}(\Omega). $$
Hence, by Step 1, it holds
$$ \partial_i (1-\varepsilon\Delta_N)^{-1}w
=
(1-\varepsilon\Delta_N)^{-1} \partial_i w
\to \partial_i w 
\quad\text{in }H^2(\Omega). $$

\smallskip
\emph{Normal derivative.}
Since $w\in H^3_{\Delta_N}(\Omega)$, then $\partial_N w\in H^2_{\Delta_D}(\Omega)$. As before
$$ \partial_N(1-\varepsilon\Delta_N)^{-1}w
=
(1-\varepsilon\Delta_D)^{-1}\partial_N w, $$
therefore
$$ (1-\varepsilon\Delta_D)^{-1}\partial_N w
\to \partial_N w\quad\text{in }H^2(\Omega). $$
\end{proof}

\section{Local well-posedness}\label{sec.local}
\subsection{Strategy of the proof}

As we mentioned in Remark \ref{rem.local-probl.}, the difficulties of the contraction argument in the local well-posedness come from the nonlinear terms that involve high order derivatives on the solutions. For this reason, we decompose the nonlinear terms as it follows: let
$$ f(V_1,V_2,V_3)={\rm Div}\left[2\xi (\Delta-a)V_1\colon V_2\left(V_3+\frac{Id}{N}\right)-(\xi+1)(\Delta-a)V_1V_2+(1-\xi)V_2(\Delta-a)V_1\right], $$
$$ g(z,V_1,V_2)=\xi(D(z)V_1+V_1D(z))+W(z)V_1-V_1W(z)-2\xi V_1\colon \nabla z \left(V_2+\frac{Id}{N}\right), $$
then 
$$ F(u,Q)=f(Q,Q,Q) + \widetilde F(u,Q),\quad G(u,Q)=g(u,Q,Q) + \widetilde G(u,Q), $$
where
$$ \widetilde F(u,Q)=F(u,Q)-f(Q,Q,Q),\quad \widetilde G(u,Q)=G(u,Q)-g(u,Q,Q). $$
We focus firstly on the linear system:
\begin{equation}\label{BE.i-lin.sys.0}
    \left\{\begin{array}{ll}
       (\partial_t-\Delta)u + \nabla \pi + \beta {\rm Div}(\Delta-a)Q=f(Q,W,W) + \widetilde F & (0,T)\times \Omega \\
       {\rm div}u=0 & (0,T)\times\Omega \\
       (\partial_t-\Delta+a)Q-\beta D(u)=g(u,W,W) + \widetilde G & (0,T)\times \Omega \\
       u=0,\quad \partial_\nu Q=0 & (0,T)\times\partial\Omega \\
       u(0)=u_0,\quad Q(0)=Q_0 & \Omega,
    \end{array}\right.
\end{equation}
where $W\in H^2(\Omega; S_0(N,\mathbb R))$ is independent of the time variable $t$. In the application, we will use $W=Q_0$. In fact, as we have already mentioned, the first entry of the functions $f$ and $g$ concentrates the highest order nonlinear terms. However, for the same reason we highlighted in Remark \ref{rem.local-probl.}, it is difficult to prove directly the existence of a solution for \eqref{BE.i-lin.sys.0}. Therefore, we pass through the approximated linear system
\begin{equation}\label{approx.lin.sys.0}
    \left\{\begin{array}{ll}
       (\partial_t-\Delta)u + \nabla \pi + \beta {\rm Div}(\Delta-a)R_\varepsilon^H Q=f(R_\varepsilon^HQ,W,W) + \widetilde F & (0,T)\times \Omega \\
       {\rm div}u=0 & (0,T)\times\Omega \\
       (\partial_t-\Delta+a)Q-\beta D(R_\varepsilon^Su)=g(R_\varepsilon^Su,W,W) + \widetilde G & (0,T)\times \Omega \\
       u=0,\quad \partial_\nu Q=0 & (0,T)\times\partial\Omega \\
       u(0)=u_0,\quad Q(0)=Q_0 & \Omega
    \end{array}\right.
\end{equation}
for $\varepsilon>0$, where $R_\varepsilon^S$ and $R_\varepsilon^H$ are the Yosida operators we introduced in the previous section. In \eqref{approx.lin.sys.0} we replaced $u$ and $Q$ respectively with $R_\varepsilon^S u$ and $R_\varepsilon^HQ$ in those terms of \eqref{BE.i-lin.sys.0} that involve high order derivatives. 

\vspace{2mm}

\noindent We resume the strategy of the proof for Theorem \ref{t.loc.ex.}:
\begin{enumerate}
    \item Prove the existence of a solution $(u_\varepsilon,Q_\varepsilon)$ for \eqref{approx.lin.sys.0} for some $T>0$;
    \item Prove a uniform bound for $(u_\varepsilon,Q_\varepsilon)$ in $Y_T(\Omega)$;
    \item Find a solution $(u,Q)\in Y_T(\Omega)$ for \eqref{BE.i-lin.sys.0} as a limit function for $(u_\varepsilon,Q_\varepsilon)$;
    \item Prove the existence of a solution for the $Q$-tensor model \eqref{BE.sys.} using the a priori estimate we found.
\end{enumerate}

\subsection{Solution for the approximated linear system}

We focus on the approximated linear system \eqref{approx.lin.sys.0}. We first consider the simpler system
\begin{equation}\label{EL.sys}
    \left\{\begin{array}{ll}
        (\partial_t-\mathbb P\Delta)u=\mathbb PF & (0,T)\times\Omega \\
        {\rm div}u=0 & (0,T)\times\Omega \\
        (\partial_t+a-\Delta)Q= G & (0,T)\times\Omega \\
        u=0,\quad \partial_\nu Q=0 & (0,T)\times\partial\Omega \\
        u(0)=u_0,\quad Q(0)=Q_0 & \Omega.
    \end{array}\right.
\end{equation}
Following the results of Section 3 of \cite{BG25}, we obtain the linear estimate:
\begin{thm}\label{t.EL-lin.es.}
Let $\Omega=\mathbb R^N,\mathbb R^N_+$ with $N\ge 2$, let $a,T>0$, let
$$ u_0\in H^1_{\mathbb P\Delta_D}\left(\Omega;\mathbb R^N\right), \quad Q_0\in H^2_{\Delta_N}\left(\Omega;S_0(N,\mathbb R)\right)$$
$$ F\in L^2\left((0,T);L^2\left(\Omega;\mathbb R^N\right)\right), \quad G\in L^2\left((0,T);H^1\left(\Omega;S_0(N,\mathbb R)\right)\right), $$
where $S_0(N,\mathbb R)$ and $H^s_A(\Omega)$ for $s\ge0$ are defined respectively in \eqref{def.S0} and in Definition \ref{def.HA} then, when $T\in(0,1)$, there is $(u,Q)\in Y_T(\Omega)$ solution for the system \eqref{EL.sys} with $u(t)\in H^2_{\mathbb P\Delta_D}(\Omega;\mathbb R^N)$ and $Q(t)\in H^3_{\Delta_N}(\Omega;S_0(N,\mathbb R))$ for a.e. $t\in(0,T)$ and 
$$ \|(u,Q)\|_{Y_T(\Omega)} \le C\left[ \|u_0\|_{H^1(\Omega)} + \|Q_0\|_{H^2(\Omega)} + \|F\|_{L^2((0,T;L^2(\Omega))} + \|G\|_{L^2((0,T);H^1(\Omega))}\right], $$
where $C=C(a,\beta)>0$.
\end{thm}
\begin{rem}\label{rem.C(T)}
Theorem \ref{t.EL-lin.es.} can be obtained by adopting the proof of Theorem 3.9 of \cite{BG25}, even though it did not consider the parameter $a>0$, nor the case $N=2$. In fact, its proof is based on the energy method and the constant $a$ can be easily treated. Moreover, when $N=2$, following the proof in \cite{BG25} it can be proved that
$$ \|(u,Q)\|_{Y_T(\Omega)} \le C(T)\left[ \|u_0\|_{H^1(\Omega)} + \|Q_0\|_{H^2(\Omega)} + \|F\|_{L^2((0,T;L^2(\Omega))} + \|G\|_{L^2((0,T);H^1(\Omega))}\right]. $$
However, when $T<1$, it holds $C(T)\lesssim1$.
\end{rem}
\begin{rem}\label{rem.press.}
It follows again from \cite{BG25} that, if $(u,Q)$ is the solution for \eqref{EL.sys}, then there is a pressure function $\pi$, unique up to additive functions $c(t)$, that solves the system 
$$ \left\{\begin{array}{ll}
    (\partial_t-\Delta)u + \nabla \pi = F & (0,T)\times \Omega \\
    (\partial_t+a-\Delta)Q = G & (0,T)\times\Omega \\
    {\rm div}u=0 & (0,T)\times \Omega \\
    u=0,\quad \partial_\nu Q=0 & (0,T)\times \partial\Omega \\
    u(0)=u_0,\quad Q(0)=Q_0 & \Omega,
\end{array}\right. $$
with $\pi(t)\in L^2_{loc}(\Omega)$ for a.e. $t\in(0,T)$, with $\nabla \pi\in L^2((0,T);L^2(\Omega;\mathbb R^N))$ and 
$$ \|\nabla \pi\|_{L^2((0,T);L^2(\Omega))}\lesssim \|u_0\|_{H^1(\Omega)} + \|Q_0\|_{H^2(\Omega)} + \|F\|_{L^2((0,T;L^2(\Omega))} + \|G\|_{L^2((0,T);H^1(\Omega))}. $$
The same consideration as in Remark \ref{rem.C(T)} about the case $N=2$ holds here.
\end{rem}
Theorem \ref{t.EL-lin.es.} gives us an a priori estimate for the simpler linear system \eqref{EL.sys}. For what concerns the nonlinear terms $f(V_1,V_2,V_3)$ and $g(z,V_1,V_2)$, the following multilinear estimates are necessary:
\begin{lem}\label{l.mult.loc.es.0}
Let $\Omega=\mathbb R^N,\mathbb R^N_+$ with $N=2,3$, let $T>0$, $k\in\mathbb N$ and $s\in\left(\frac{N}{2}-1,1\right]$, let $v\in X^1_T(\Omega)$ and $w_j\in L^\infty((0,T);H^{s+1}(\Omega))$ for $j=1,\ldots, k$, then it holds 
$$ \|\nabla vw_1\cdots w_k\|_{L^2((0,T);H^1(\Omega))}\lesssim \|\nabla v\|_{L^2((0,T);H^1(\Omega))} \prod_{j=1}^k\|w_j\|_{L^\infty((0,T);H^{s+1}(\Omega))}. $$
In particular, for any $(v,w)\in Y_T(\Omega)$ and for any $w_1,w_2\in L^\infty((0,T);H^{s+1}(\Omega))$, it holds
$$ \|f(w,w_1,w_2)\|_{L^2((0,T);L^2(\Omega))} + \|g(v,w_1,w_2)\|_{L^2((0,T);H^1(\Omega))} $$
$$ \lesssim \|(\nabla v,w)\|_{L^2((0,T);H^1(\Omega)\times H^3(\Omega))}\|w_1\|_{L^\infty((0,T);H^{s+1}(\Omega))}\left(1 + \|w_2\|_{L^\infty((0,T);H^{s+1}(\Omega))}\right). $$
\end{lem}
\begin{proof}\hfill\\
We divide the proof of the first inequality considering separately the quantities:
$$ \|\nabla vw_1\cdots w_k\|_{L^2_TH^1}= \|\nabla vw_1\cdots w_k\|_{L^2_TL^2} + $$
$$ + \|\nabla^2 v\cdot w_1\cdots w_k\|_{L^2_TL^2}  + \sum_{\ell=1}^k\|\nabla vw_1\cdots w_{\ell-1}\nabla w_\ell w_{\ell+1}\cdots w_k\|_{L^2_TL^2}. $$
For what concerns the first term
$$ \|\nabla vw_1\cdots w_k\|_{L^2_TL^2}\le \|\nabla v\|_{L^2_TL^2}\prod_{j=1}^k\|w_j\|_{L^\infty_TL^\infty} $$
and thanks to the Sobolev embedding $H^{s+1}(\Omega)\hookrightarrow L^\infty(\Omega)$ for $s>\frac{N}{2}-1$ we get
$$ \|\nabla vw_1\cdots w_k\|_{L^2_TL^2}\lesssim \|\nabla v\|_{L^2_TL^2} \prod_{j=1}^k\|w_j\|_{L^\infty_T H^{s+1}}. $$
With the same strategy, it holds
$$ \|\nabla^2 v\cdot w_1\cdots w_k\|_{L^2_TL^2}\lesssim \|\nabla^2 v\|_{L^2_TL^2} \prod_{j=1}^k\|w_j\|_{L^\infty_TH^{s+1}}. $$
Finally, for the last term, we distinguish the cases $k=1$ and $k>1$. In the first case
$$ \|\nabla v\nabla w_1\|_{L^2_TL^2}\le \left\|\|\nabla v\|_{L^\frac{N}{s}(\Omega)}\|\nabla w_1\|_{L^\frac{2N}{N-2s}(\Omega)}\right\|_{L^2((0,T))}. $$
By the Sobolev embedding $H^1(\Omega)\hookrightarrow L^\frac{N}{s}(\Omega)$ and $H^{s}(\Omega)\hookrightarrow L^\frac{2N}{N-2s}(\Omega)$ for $s>\frac{N}{2}-1$, it holds 
$$ \left\|\|\nabla v\|_{L^\frac{N}{s}(\Omega)}\|\nabla w_1\|_{L^\frac{2N}{N-2s}(\Omega)}\right\|_{L^2((0,T))} \lesssim \|\nabla v\|_{L^2_TH^1}\|w_1\|_{L^\infty_TH^{s+1}}. $$
The case $k\ge2$ can be done similarly using the embedding $H^{s+1}(\Omega)\hookrightarrow L^\infty(\Omega)$ for $s>\frac{N}{2}-1$. The inequalities for $f$ and $g$ follows from the previous inequality and the Sobolev embeddings.
\end{proof}
We are now ready to prove the existence of a solution for \eqref{approx.lin.sys.0}.
\begin{prop}\label{p.loc.ex-approx.}
Let $\Omega=\mathbb R^N,\mathbb R^N_+$ with $N=2,3$, let $a>0$, $\beta\in\mathbb R$ and $T\in(0,1)$, let $W\in H^2(\Omega;S_0(N,\mathbb R))$, $u_0\in H^1_{\mathbb P\Delta_D}(\Omega;\mathbb R^N)$ and $Q_0\in H^2_{\Delta_N}(\Omega;S_0(N,\mathbb R))$ where $S_0(N,\mathbb R)$ and $H^s_A(\Omega)$ for $s\ge 0$ are defined in \eqref{def.S0} and in Definition \ref{def.HA}, let
$$ \widetilde F\in L^2\left((0,T);L^2\left(\Omega;\mathbb R^N\right)\right),\quad \widetilde G\in L^2\left((0,T);H^1\left(\Omega;S_0(N,\mathbb R)\right)\right),\quad W\in H^2\left(\Omega;S_0(N,\mathbb R)\right), $$
then for any $\varepsilon\in(0,1)$ the system \eqref{approx.lin.sys.0} admits a solution $(u_\varepsilon,Q_\varepsilon)\in Y_T(\Omega)$, where $Y_T$ is defined in \eqref{def.Y}, with $u_\varepsilon(t)\in H^2_{\mathbb P\Delta_D}(\Omega;\mathbb R^N)$ and $Q_\varepsilon(t)\in H^3_{\Delta_N}(\Omega;S_0(N,\mathbb R))$ for a.e. $t\in(0,T)$.
\end{prop}
\begin{proof}\hfill\\
Let us consider the map $\Phi(z,V)=(u,Q)$, where $(u,Q)$ solves the system
\begin{equation}\label{approx.sys-contr.}
    \left\{\begin{array}{ll}
       (\partial_t-\mathbb P\Delta)u= - \beta \mathbb P {\rm Div}(\Delta-a) R_\varepsilon^H V + \mathbb Pf(R_\varepsilon^HV,W,W) + \mathbb P\widetilde F & (0,T)\times \Omega \\
       {\rm div}u=0 & (0,T)\times \Omega \\
       (\partial_t-\Delta+a)Q = \beta D(R_\varepsilon^Sz) + g(R_\varepsilon^Sz,W,W) + \widetilde G & (0,T)\times \Omega \\
       u=0,\quad \partial_\nu Q=0 & (0,T)\times\partial\Omega \\
       u(0)=u_0,\quad Q(0)=Q_0 & \Omega.
    \end{array}\right.
\end{equation}
Our purpose is to apply the Banach Fixed Point Theorem on $\Phi$ choosing a proper Banach space for $(z,V)$ and $(u,Q)$. Let 
$$ Z=\{(z,V)\in Y_T(\Omega)\mid z(0)=u_0,\quad V(0)=Q_0\}. $$
We know from Theorem \ref{t.EL-lin.es.} that a solution $(u,Q)\in Z$ exists when
$$ -\beta {\rm Div}(\Delta-a) R_\varepsilon^H V + f(R_\varepsilon^HV,W,W) + \widetilde F\in L^2((0,T);L^2(\Omega;\mathbb R^N)), $$
$$ \beta D(R_\varepsilon^Sz) + g(R_\varepsilon^Sz,W,W) + \widetilde G\in L^2((0,T);H^1(\Omega;S_0(N,\mathbb R))).   $$
In particular, if we take $(z,V)\in Z$, such a property follows from Lemma \ref{l.mult.loc.es.0}. So, the map $\Phi\colon Z\to Z$ is well-defined. More precisely, Lemma \ref{l.mult.loc.es.0} states that 
$$ \|f(R_\varepsilon^HV,W,W)\|_{L^2_TL^2} + \|g(R_\varepsilon^Sz,W,W)\|_{L^2_TH^1} $$
$$ \lesssim \left(\|R_\varepsilon^Sz\|_{L^2_TH^2} + \|R_\varepsilon^HV\|_{L^2_TH^3}\right)\left(1 +  \|W\|_{H^2(\Omega)}^2\right), $$
and, since $\varepsilon<1$, thanks to Lemma \ref{l.res.es.} it holds
$$ \|R^S_\varepsilon z\|_{L^2_TH^2} + \|R^H_\varepsilon V\|_{L^2_TH^3} $$
$$ \lesssim \varepsilon^{-1/2}\left(\|z\|_{L^2_TH^1} + \|V\|_{L^2_TH^2}\right) $$
$$ \le \varepsilon^{-1/2}T^{1/2}\left(\|z\|_{L^\infty_TH^1} + \|V\|_{L^\infty_TH^2}\right) \le \varepsilon^{-1/2}T^{1/2}\left(\|z\|_{X^1_T(\Omega)} + \|V\|_{X^2_T(\Omega)}\right). $$
Similarly
$$ \|{\rm Div}(\Delta-a)R_\varepsilon^H V\|_{L^2_TL^2} + \|D(R_\varepsilon^S z)\|_{L^2_TH^1} $$
$$ \lesssim \varepsilon^{-1/2}T^{1/2}\left(\|z\|_{L^2_TH^1} + \|V\|_{L^2_TH^2}\right). $$
So, applying the linear estimate from Theorem \ref{t.EL-lin.es.}, it holds
\begin{equation}\label{proof.nl.es.tot.}
    \|\Phi(z,V)\|_{Y_T(\Omega)}\lesssim \|u_0\|_{H^1(\Omega)} + \|Q_0\|_{H^2(\Omega)} + \varepsilon^{-1/2}T^{1/2}\left(1 +  \|W\|_{H^2(\Omega)}^2\right)\|(z,V)\|_{Y_T(\Omega)}.
\end{equation}
In order to apply the Banach Fixed Point Theorem, we need to prove that $\Phi$ is a contraction in a proper subspace of $Z$: let us consider
$$ Z_{\omega}=\{(v,W)\in Z\mid \|(v,W)\|_{Y_T(\Omega)}\le \omega\}. $$
Firstly, we show that $\Phi\colon Z_\omega\to Z_\omega$. In fact, letting $(u,Q)=\Phi(z,V)$, by the estimate \eqref{proof.nl.es.tot.} there exists $C_0>0$ such that
$$ \|(u,Q)\|_{Y_T(\Omega)}\le C_0(\|u_0\|_{H^1(\Omega)} + \|Q_0\|_{H^2(\Omega)}) + C_0\varepsilon^{-1/2}T^{1/2}\omega\left(1 + \|W\|_{H^2(\Omega)}^2\right). $$
We then take $\omega$ such as
$$ C_0(\|u_0\|_{H^1(\Omega)} + \|Q_0\|_{H^2(\Omega)})=\frac{\omega}{2}  $$
and $T_\varepsilon>0$ such that 
$$ C_0\varepsilon^{-1/2}T_\varepsilon^{1/2}(1+\|W\|_{H^2(\Omega)}^2)\le \frac{1}{2}, $$
so that 
$$ \|(u,Q)\|_{Y_{T_\varepsilon}(\Omega)}\le \omega $$
and therefore $\Phi\colon Z_\omega\to Z_\omega$. Let us take now $(z_1,V_1),(z_2,V_2)\in Z_\omega$ and the correspondent $(u_j,Q_j)=\Phi(z_j,V_j)$ for $j=1,2$, then $(u_1-u_2,Q_1-Q_2)$ solves the system
$$ \left\{\begin{array}{ll}
    (\partial_t-\mathbb P\Delta)(u_1-u_2)= - \beta\mathbb P{\rm Div}(\Delta-a)(R_\varepsilon^H(V_1-V_2)) + \mathbb P f(R_\varepsilon^H(V_1-V_2),W,W) & (0,T_\varepsilon)\times\Omega \\
    {\rm div}(u_1-u_2)=0 & (0,T_\varepsilon)\times\Omega \\
    (\partial_t+a-\Delta)(Q_1-Q_2) = \beta D(R_\varepsilon^S(z_1-z_2)) + g(R_\varepsilon^S(z_1-z_2),W,W)& (0,T_\varepsilon)\times \Omega \\
    u_1-u_2=0,\quad \partial_\nu (Q_1-Q_2)=0 & (0,T_\varepsilon)\times\partial\Omega \\
    (u_1-u_2)(0)=0,\quad (Q_1-Q_2)(0)=0 & \Omega.
\end{array}\right. $$
So, we obtain as before that there is $C_1>0$ such that
$$  \|(u_1-u_2, Q_1-Q_2)\|_{Y_{T_\varepsilon}(\Omega)}\le C_1\varepsilon^{-1/2}{T_\varepsilon}^\frac{1}{2}(1 +  \|W\|_{H^2(\Omega)}^2)\|(z_1-z_2,V_1-V_2)\|_{Y_{T_\varepsilon}(\Omega)}. $$
Therefore, $\Phi\colon Z_\omega\to Z_\omega$ is a contraction choosing $T_\varepsilon$ sufficiently small such that 
$$ C_1\varepsilon^{-1/2}T_\varepsilon^\frac{1}{2}(1 +  \|W\|_{H^2(\mathbb R^3)}^2)<1.  $$
Therefore, we get a solution for the approximated system \eqref{approx.lin.sys.0} in $(0,T_\varepsilon)$. 

Finally, it may happen that $T_\varepsilon<T$, so we need to extend our solution to $(0,T)$. To do so, we can apply the previous argument changing the initial conditions: as before, we can find a solution for the system \eqref{approx.sys-contr.} with initial conditions $u(T_\varepsilon-\delta)$ and $Q(T_\varepsilon-\delta)$ for some $\delta>0$. We notice that the choice of $T_\varepsilon$ does not depend on the initial conditions, so we can extend the solution in $2T_\varepsilon-\delta$. Therefore, with a finite number of steps we can extend the solution in $(0,T)$. 
\end{proof}
\begin{rem}
From now on, we apply Proposition \ref{p.loc.ex-approx.} with $W=Q_0$. We introduced the function $W$ in the system \eqref{approx.lin.sys.0} to clarify the proof of Proposition \ref{p.loc.ex-approx.}: at the end of the proof, we extended the solution $(u_\varepsilon,Q_\varepsilon)$ from $(0,T_\varepsilon)$ to $(0,T)$ using the fact that $T_\varepsilon$ depends on $W$ but not on the initial conditions. So, if we set $W=Q_0$, the previous argument would be confusing for the reader.
\end{rem}

\subsection{Energy method and cancellations}

In the previous section we proved the existence of a solution $(u_\varepsilon,Q_\varepsilon)$ for the approximated linear system 
\begin{equation}\label{approx.lin.sys.2}
    \left\{\begin{array}{ll}
       (\partial_t-\mathbb P\Delta)u_\varepsilon +  \beta \mathbb P {\rm Div}(\Delta-a) R_\varepsilon^H Q_\varepsilon = \mathbb Pf(R_\varepsilon^HQ_\varepsilon,Q_0,Q_0) + \mathbb P\widetilde F & (0,T)\times \Omega \\
       {\rm div}u_\varepsilon=0 & (0,T)\times \Omega \\
       (\partial_t-\Delta+a)Q_\varepsilon - \beta D(R_\varepsilon^Su_\varepsilon) =g(R_\varepsilon^Su_\varepsilon,Q_0,Q_0) + \widetilde G & (0,T)\times \Omega \\
       u_\varepsilon=0,\quad \partial_\nu Q_\varepsilon=0 & (0,T)\times\partial\Omega \\
       u_\varepsilon(0)=u_0,\quad Q_\varepsilon(0)=Q_0 & \Omega,
    \end{array}\right.
\end{equation}
for any $T\in(0,1)$. Our next purpose is to show that such a solution is uniformly bounded in $Y_T(\Omega)$, that is we can find $C>0$ independent of $\varepsilon$ such that 
$$ \|(u_\varepsilon,Q_\varepsilon)\|_{Y_T(\Omega)}\le C\left[\|u_0\|_{H^1(\Omega)} + \|Q_0\|_{H^2(\Omega)} + \|\widetilde F\|_{L^2_TL^2} + \|\widetilde G\|_{L^2_TH^1}\right]. $$
First, we need a technical lemma:
\begin{lem}\label{l.matrix-trace}
Let $A,B$ be symmetric matrices, then 
\begin{itemize}
    \item If $C$ is a symmetric matrix, it holds
    $$ (AB)\colon C=(CA)\colon B=(BC)\colon A; $$
    \item If $C$ is an anti-symmetric matrix, it holds
    $$ (AB)\colon C=-(CA)\colon B, $$
    $$ (AB)\colon C = - (BC)\colon A. $$
\end{itemize}
\end{lem}
\begin{proof}\hfill\\
For what concerns the first part, it is sufficient to notice that for any matrices $A,B,C$ it holds
\begin{equation}\label{aux.id.matrix.}
{\rm tr}(ABC)={\rm tr}(CAB)={\rm tr}(BCA).
\end{equation}
Conversely, when $C$ is anti-symmetric
$$ (AB)\colon C={\rm tr}(C^TAB)=-{\rm tr}(CAB). $$
Thanks to \eqref{aux.id.matrix.}
$$ -{\rm tr}(CAB)=-{\rm tr}(BCA)=-(CA)\colon B. $$
Similarly
$$ (AB)\colon C={\rm tr}(C^TAB)=-{\rm tr}(CAB)=-{\rm tr}(ABC)=-(BC)\colon A. $$
\end{proof}
We state here some remarks that we repeatedly use in the next calculations:
\begin{rem}\label{rem.help-calc.}
Let $u\in H^2_{\mathbb P\Delta_D}(\Omega)$.
\begin{itemize}
    \item Since $H^{2}_{\mathbb{P}\Delta_D}(\Omega) \subset J_2(\Omega)$, it holds
    $$ \int_{\Omega} \mathbb{P} f \cdot udx = \int_{\Omega} f \cdot udx
    \qquad \forall\, f \in L^2(\Omega;\mathbb{R}^N). $$
    Moreover,
    \begin{equation}\label{help-rem.Id-u}
    Id\colon \partial_j^\alpha\nabla u={\rm div}\partial_j^\alpha u=0\quad j=1,\ldots, N,\quad \alpha=0,1.
    \end{equation}
    \item If $V\in S_0(N,\mathbb R)$
    \begin{equation}\label{help-rem.nablau-Q}
        \nabla u\colon V = D (u)\colon V,
    \end{equation}
    \begin{equation}\label{help-rem.Id-Q}
     Id\colon V={\rm tr}V=0.   
    \end{equation}
    In the applications, we will consider $V=\nabla^\alpha Q$ for $\alpha=0,1,2,3$ with $Q\in H^3(\Omega;S_0(N,\mathbb R))$. 
\end{itemize}
\end{rem}
As discussed above, applying the energy method to the system \eqref{approx.lin.sys.0} reveals cancellations in both the linear and nonlinear terms. The following lemmas highlight these phenomena through a term-by-term analysis. These cancellations depend on the domain we consider: when $\Omega=\mathbb R^N$ they yield to stronger results with respect to the case $\Omega=\mathbb R^N_+$. However, the strategy of the proof is the same. As the half-space case is more involved, we restrict ourselves to stating and proving the result for $\Omega=\mathbb{R}^N_+$.
\begin{lem}\label{l.en-can.1}
Let $\Omega=\mathbb R^N,\mathbb R^N_+$ with $N=2,3$, let $a,T>0$, $(u,Q)\in Y_T(\Omega)$ such that, for a.e. $t\in(0,T)$, it holds
$$ u(t),\partial_j^2 u(t)\in H^2_{\mathbb P\Delta_D}(\Omega;\mathbb R^N) \quad j=1,\ldots, N , $$
$$ Q(t),\Delta Q(t)\in H^3_{\Delta_N}(\Omega;S_0(N,\mathbb R)), $$
then for a.e. $t\in(0,T)$ it holds
\begin{equation}\label{en-can.1-0}
    \int_{\Omega}\mathbb P{\rm Div}(\Delta-a)Q\cdot u-D (u)\colon (a-\Delta)Qdx=0,
\end{equation}
\begin{equation}\label{en-can.1-i}
    \int_{\Omega}\mathbb P{\rm Div}(\Delta-a)Q\cdot (-\partial_i^2)u-\partial_i D(u)\colon \partial_i(a-\Delta)Qdx=0\quad i<N
\end{equation}
\begin{equation}\label{en-can.1-N}
    \begin{aligned}
        \left|\int_0^t\int_{\Omega}\mathbb P{\rm Div}(\Delta-a)Q\cdot (-\partial_N^2)u-\partial_ND(u)\colon \partial_N(a-\Delta)Qdxd\tau\right| \\
        \le \|\nabla^\prime\nabla u\|_{L^2((0,t);L^2(\Omega)}\|Q\|_{X^2_t(\Omega)} + \|\nabla^\prime Q\|_{L^2((0,t);H^2(\Omega))}\|\nabla u\|_{X^0_t(\Omega)}.
    \end{aligned}
\end{equation}
\end{lem}
\begin{proof}\hfill\\
As anticipated, we focus on the case $\Omega=\mathbb R^N_+$. Since $u,\partial_j^2u\in H^2_{\mathbb P\Delta_D}(\mathbb R^N_+)$ for $j=1,\ldots,N$, as mentioned in Remark \ref{rem.help-calc.} we can remove the Helmholtz projection from the calculation. The identity \eqref{en-can.1-0} follows from the condition $u=0$ on $\mathbb R^N_0$ combined with \eqref{help-rem.nablau-Q}. The second identity \eqref{en-can.1-i} also follows from the boundary condition $\partial_i u=0$ on $\mathbb R^N_0$ when $i<N$. For the third result, let us write more precisely the left hand side of \eqref{en-can.1-N}:
$$ -\int_0^t\int_{\mathbb R^N_+}{\rm Div}(\Delta-a)Q\cdot \partial_N^2u+\partial_ND(u)\colon \partial_N(a-\Delta)Qdxd\tau $$
$$ =- \int_0^t\int_{\mathbb R^N_+}{\rm Div}(\Delta-a)Q\cdot \partial_N^2u+\partial_N\nabla u\colon \partial_N(a-\Delta)Qdxd\tau $$
$$ = -\sum_{j,k=1}^N\int_0^t\int_{\mathbb R^N_+}\partial_k(\Delta-a)Q_{jk}\partial_N^2u_j+\partial_N\partial_k u_j\partial_N(a-\Delta)Q_{jk}dxd\tau, $$
where in the first equality we applied \eqref{help-rem.nablau-Q} with $V=\partial_N(a-\Delta)Q$. When $k=N$, we have the cancellation:
$$ \int_0^t\int_{\mathbb R^N_+}\partial_N(\Delta-a)Q_{jN}\partial_N^2u_j+\partial_N^2 u_j\partial_N(a-\Delta)Q_{jN}dxd\tau=0 \quad j=1,\ldots,N. $$
For the remaining terms, that is $k<N$, we can use the H\"older inequality:
$$ \left|\int_0^t\int_{\mathbb R^N_+}\partial_k(\Delta-a)Q_{jk}\partial_N^2u_j+\partial_N\partial_k u_j\partial_N(a-\Delta)Q_{jk}dxd\tau\right| $$
$$ \lesssim  \|\partial_kQ\|_{L^2_tH^2}\|\partial_N^2u\|_{L^2_tL^2} + \|\partial_N\partial_k u\|_{L^2_tL^2}\|\partial_NQ\|_{L^2_tH^2} $$
$$ \le \|\nabla^\prime Q\|_{L^2_tH^2}\|\nabla u\|_{L^2_tH^1} + \|\nabla^\prime \nabla u\|_{L^2_tL^2}\|Q\|_{L^2_tH^3}. $$
\end{proof}

\begin{lem}\label{l.en-can.2}
Let $\Omega=\mathbb R^N,\mathbb R^N_+$ with $N=2,3$, let $a,T>0$, $(u,Q)\in Y_T(\Omega)$ as in Lemma \ref{l.en-can.1}, let $Q_0\in H^2_{\Delta_N}(\Omega;S_0(N,\mathbb R))$, then for 
a.e. $t\in(0,T)$ it holds
\begin{equation}\label{en-can.2-0}
    \int_{\Omega}\mathbb P{\rm Div}\left((\Delta-a)Q\colon Q_0\left(Q_0+\frac{Id}{N}\right)\right)\cdot u-\nabla u\colon Q_0\left(Q_0+\frac{Id}{N}\right)\colon (a-\Delta)Qdx=0,
\end{equation}
\begin{equation}\label{en-can.2-i}
    \begin{aligned}
        & \left|\int_0^t\int_{\Omega}\mathbb P{\rm Div}\left((\Delta-a)Q\colon Q_0\left(Q_0+\frac{Id}{N}\right)\right)\cdot (-\partial_i^2)u-\partial_i\left(\nabla u\colon Q_0\left(Q_0+\frac{Id}{N}\right)\right)\colon \partial_i(a-\Delta)Qdxd\tau\right| \\
        & \lesssim t^\gamma \|Q_0\|_{H^2(\Omega)}^2\|Q\|_{X^2_t(\Omega)}\left(\|u\|_{L^\infty((0,t);H^1(\Omega))} + \|\nabla u\|_{X^0_t(\Omega)}\right)\quad i<N,
    \end{aligned}
\end{equation}
\begin{equation}\label{en-can.2-N}
    \begin{aligned}
        & \left|\int_0^t\int_{\Omega}\mathbb P{\rm Div}\left((\Delta-a)Q\colon Q_0\left(Q_0+\frac{Id}{N}\right)\right)\cdot (-\partial_N^2)u-\partial_N\left(\nabla u\colon Q_0\left(Q_0+\frac{Id}{N}\right)\right)\colon \partial_N (a-\Delta)Qdxd\tau\right| \\
        & \lesssim t^\gamma \|Q_0\|_{H^2(\Omega)}^2\|Q\|_{X^2_t(\Omega)}\left(\|u\|_{L^\infty((0,t);H^1(\Omega))} + \|\nabla u\|_{X^0_t(\Omega)}\right) + \\
        & + \|Q_0\|_{H^2(\Omega)}^2\left(\|\nabla^\prime \nabla u\|_{L^2((0,t);L^2(\Omega))}\|Q\|_{X^2_t(\Omega)} + \|\nabla u\|_{X^0_t(\Omega)}\|\nabla^\prime Q\|_{L^2((0,t);H^2(\Omega))}\right),
    \end{aligned}
\end{equation}
for some $\gamma>0$.
\end{lem}
\begin{proof}\hfill\\
As before, we can focus on the estimates without considering the operator $\mathbb P$. The identity \eqref{en-can.2-0} follows from the Dirichlet conditions on $u$ with \eqref{help-rem.Id-u} and \eqref{help-rem.Id-Q}. For what concerns \eqref{en-can.2-i}, thanks to the regularity of $Q$ and $Q_0$ and to the Dirichlet conditions on $\partial_iu$ for $i<N$, it holds
$$ -\int_0^t\int_{\mathbb R^N_+}{\rm Div}\left((\Delta-a)Q\colon Q_0\left(Q_0+\frac{Id}{N}\right)\right)\cdot \partial_i^2u dxd\tau $$
$$ = \int_0^t\int_{\mathbb R^N_+}\partial_i{\rm Div}\left((\Delta-a)Q\colon Q_0\left(Q_0+\frac{Id}{N}\right)\right)\cdot\partial_iu dx d\tau $$
$$ = -\int_0^t\int_{\mathbb R^N_+}\partial_i\left((\Delta-a)Q\colon Q_0\left(Q_0+\frac{Id}{N}\right)\right)\colon \nabla \partial_iu dxd\tau. $$
Thanks to \eqref{help-rem.Id-u} and \eqref{help-rem.Id-Q}, the terms involving the identity matrix can be removed. So, it is sufficient to study the quantity
$$ -\int_0^t\int_{\mathbb R^N_+} \partial_i((\Delta-a)Q\colon Q_0Q_0)\colon \nabla \partial_i u + \partial_i(\nabla u\colon Q_0Q_0)\colon \partial_i(a-\Delta)Q dxd\tau. $$
We notice that, when $\partial_i$ acts on $(\Delta-a)Q$ in the first term and on $\nabla u$ in the second, there is a cancellation:
$$ -\int_0^t\int_{\mathbb R^N_+} (\partial_i(\Delta-a)Q\colon Q_0)\cdot (Q_0\colon \nabla \partial_i u) + (\partial_i\nabla u\colon Q_0)\cdot (Q_0\colon \partial_i(a-\Delta)Q) dxd\tau=0. $$
The remaining terms can be bounded by Sobolev and H\"older inequalities:
$$ \left|\int_0^t\int_{\mathbb R^N_+} (\Delta-a)Q\partial_iQ_0Q_0 \nabla \partial_i u dx d\tau\right| + \left|\int_0^t\int_{\mathbb R^N_+} \nabla u\partial_iQ_0Q_0\partial_i (a-\Delta)Q dx d\tau\right| $$
$$ \le \|Q_0\|_{L^\infty(\mathbb R^N_+)}\|\partial_iQ_0\|_{L^6(\mathbb R^N_+)}\|(\Delta-a)Q\|_{L^2_tL^3}\|\nabla \partial_iu\|_{L^2_tL^2} +  $$
$$ + \|Q_0\|_{L^\infty(\mathbb R^N_+)}\|\partial_iQ_0\|_{L^6(\mathbb R^N_+)}\|\partial_i(\Delta-a)Q\|_{L^2_tL^2}\|\nabla u\|_{L^2_tL^3}$$
Firstly, we notice that $H^2(\mathbb R^N_+)\hookrightarrow L^\infty(\mathbb R^N_+)$ and $H^1(\mathbb R^N_+)\hookrightarrow L^6(\mathbb R^N_+)$ for $N=2,3$, so
$$ \|Q_0\|_{L^\infty(\mathbb R^N_+)}\|\partial_iQ_0\|_{L^6(\mathbb R^N_+)}\lesssim \|Q_0\|^2_{H^2(\mathbb R^N_+)}. $$
Moreover, by interpolation
$$ \|(\Delta-a)Q\|_{L^3(\mathbb R^N_+)}\lesssim \|(\Delta-a)Q\|^{1/2}_{L^2(\mathbb R^N_+)}\|(\Delta-a)Q\|^{1/2}_{L^6(\mathbb R^N_+)}, $$
$$ \|\nabla u\|_{L^3(\mathbb R^N_+)}\lesssim \|\nabla u\|_{L^2(\mathbb R^N_+)}^{1/2}\|\nabla u\|_{L^6(\mathbb R^N_+)}^{1/2}, $$
and therefore
$$ \|(\Delta-a)Q\|_{L^2_tL^3}\lesssim \|Q\|^{1/2}_{L^\infty_tH^2}\|Q\|_{L^1_tH^3}^{1/2} \le t^{1/4}\|Q\|^{1/2}_{L^\infty_tH^2}\|Q\|_{L^2_tH^3}^{1/2}\le t^{1/4}\|Q\|_{X^2_t(\mathbb R^N_+)}, $$
$$ \|\nabla u\|_{L^2_tL^3}\lesssim \|u\|^{1/2}_{L^\infty_tH^1}\|\nabla u\|_{L^1_tH^1}^{1/2} \le t^{1/4}\|u\|^{1/2}_{L^\infty_tH^1}\|\nabla u\|_{L^2_tH^1}^{1/2}\lesssim t^{1/4}\left[\|u\|_{L^\infty_tH^1} + \|\nabla u\|_{L^2_tH^1}\right]. $$
Finally, we consider the inequality \eqref{en-can.2-N}. Let us write more in details the left hand side  (the term with the $Id$ matrix can be treated as before) :
\begin{equation}\label{proof.can.lem.1}
    \begin{aligned}
        -{\rm Div}((\Delta-a)Q\colon Q_0 Q_0)\cdot \partial_N^2u-\partial_N(\nabla u\colon Q_0Q_0)\colon \partial_N (a-\Delta)Q \\
        = -\sum_{j,k,\ell,h=1}^N \partial_k((\Delta-a)Q_{\ell h}Q_0^{h\ell}Q_0^{jk})\partial_N^2u_j + \partial_N(\partial_ku_jQ_0^{jk}Q_0^{h\ell})\partial_N(a-\Delta)Q_{\ell h}.
    \end{aligned}
\end{equation}
We first consider the case where $\partial_k$ acts on $(\Delta - a)Q$ in the first term, and $\partial_N$ acts on $\nabla u$ in the second:
$$ -\int_0^t\int_{\mathbb R^N_+}\partial_k(\Delta-a)Q_{\ell h}Q_0^{h\ell}Q_0^{jk}\partial_N^2u_j +\partial_N\partial_ku_jQ_0^{jk}Q_0^{h\ell}\partial_N(a-\Delta)Q_{\ell h}dxd\tau. $$
When $k=N$, the two terms are opposite and there is a cancellation. Conversely, when $k<N$, the thesis can be easily verified  using the Sobolev embedding and the H\"older inequality as we did in the proof of Lemma \ref{l.en-can.1}.
\end{proof}

It remains to estimate the terms
$$ {\rm Div}\left(-(\xi+1)(\Delta-a)QQ_0+(1-\xi)Q_0(\Delta-a)Q\right), $$
and
$$ \xi(D(u)Q_0+Q_0D(u))+W(u)Q_0-Q_0W(u), $$
where we recall that $D(u)$ and $W(u)$ are respectively the symmetric and the anti-symmetric part of the gradient of $u$, that is
$$ D(u)=\frac{1}{2}(\nabla u + \nabla^Tu), \quad W(u)=\frac{1}{2}(\nabla u - \nabla^Tu). $$
In particular
$$ \nabla u= D(u) + W(u) $$
and therefore
\begin{equation}\label{gr.matrix-id.}
    \nabla^Tu=(\nabla u)^T=D(u)-W(u).
\end{equation}
We divide the proof for the remaining terms into two parts. In fact
$$ -(\xi+1)(\Delta-a)Q Q_0 + (1-\xi)Q_0(\Delta-a)Q $$
$$ = (Q_0(\Delta-a)Q - (\Delta-a)Q Q_0) - \xi((\Delta-a)Q Q_0 + Q_0(\Delta-a)Q), $$
So, we study separately the quantities
$$ \int_0^t\int_{\Omega}-\mathbb P{\rm Div}\left((\Delta-a)Q Q_0+Q_0(\Delta-a)Q\right)\cdot u+(D(u)Q_0+Q_0D(u))\colon (a-\Delta)Qdxd\tau $$
and
$$ \int_0^t\int_{\Omega}\mathbb P{\rm Div}\left(Q_0(\Delta-a)Q - (\Delta-a)Q Q_0\right)\cdot u+(W(u)Q_0 - Q_0W(u))\colon (a-\Delta)Qdxd\tau. $$
\begin{lem}\label{l.en-can.3}
Let $\Omega=\mathbb R^N,\mathbb R^N_+$ with $N=2,3$, let $a,T>0$, $(u,Q)\in Y_T(\Omega)$  as in Lemma \ref{l.en-can.1}, let $Q_0\in H^2_{\Delta_N}(\Omega;S_0(N,\mathbb R))$, then for 
a.e. $t\in(0,T)$ it holds
\begin{equation}\label{en-can.3-0}
    \int_{\Omega}-\mathbb P{\rm Div}\left((\Delta-a)Q Q_0+Q_0(\Delta-a)Q\right)\cdot u+(D(u)Q_0+Q_0D(u))\colon (a-\Delta)Qdx=0,
\end{equation}
\begin{equation}\label{en-can.3-i}
    \begin{aligned}
        & \left|\int_0^t\int_{\Omega}-\mathbb P{\rm Div}\left((\Delta-a)Q Q_0+Q_0(\Delta-a)Q\right)\cdot (-\partial_i^2)u+\partial_i(D(u)Q_0+Q_0D(u))\colon \partial_i(a-\Delta)Qdxd\tau\right| \\
        & \lesssim t^\gamma \|Q_0\|_{H^2(\Omega)}\|Q\|_{X^2_t(\Omega)}\left(\|u\|_{L^\infty((0,t);H^1(\Omega))} + \|\nabla u\|_{X^0_t(\Omega)}\right)\quad i<N
    \end{aligned}
\end{equation}
\begin{equation}\label{en-can.3-N}
    \begin{aligned}
        & \left|\int_0^t\int_{\Omega}-\mathbb P{\rm Div}\left((\Delta-a)Q Q_0+Q_0(\Delta-a)Q\right)\cdot (-\partial_N^2)u+\partial_N(D(u)Q_0+Q_0D(u))\colon \partial_N(a-\Delta)Qdxd\tau\right| \\
        & \lesssim t^\gamma \|Q_0\|_{H^2(\Omega)}\|Q\|_{X^2_t(\Omega)}\left(\|u\|_{L^\infty((0,t);H^1(\Omega))} + \|\nabla u\|_{X^0_t(\Omega)}\right) + \\
        & + \|Q_0\|_{H^2(\Omega)}\left(\|\nabla^\prime \nabla u\|_{L^2((0,t);L^2(\Omega))}\|Q\|_{X^2_t(\Omega)} + \|\nabla u\|_{X^0_t(\Omega)}\|\nabla^\prime Q\|_{L^2((0,t);H^2(\Omega))}\right),
    \end{aligned}
\end{equation}
for some $\gamma>0$.
\end{lem}
\begin{proof}\hfill\\
Using integration by parts and \eqref{help-rem.nablau-Q}
$$ \int_{\mathbb R^N_+}-{\rm Div}\left((\Delta-a)Q Q_0+Q_0(\Delta-a)Q\right)\cdot u dx $$
$$ = \int_{\mathbb R^N_+}\left((\Delta-a)Q Q_0+Q_0(\Delta-a)Q\right)\colon D(u)  dx. $$
Moreover, Lemma \ref{l.matrix-trace} allows us to rewrite
$$ \int_{\mathbb R^N_+}\left((\Delta-a)Q Q_0+Q_0(\Delta-a)Q\right)\colon D(u)  dx = \int_{\mathbb R^N_+}\left(Q_0 D(u) + D(u)Q_0\right)\colon (\Delta-a)Q dx, $$
which implies the identity \eqref{en-can.3-0}. The estimate \eqref{en-can.3-i} can be done as before exploiting the boundary condition $\partial_i u=0$ for $i<N$ on $\mathbb R^N_0$. For the estimate \eqref{en-can.3-N}, we use the identity \eqref{gr.matrix-id.}:
$$ Q_0D(u)+D(u)Q_0=Q_0\nabla^Tu + \nabla^TuQ_0 + (Q_0W(u)+W(u)Q_0). $$
The second term is anti-symmetric, so 
$$ \partial_N(Q_0D(u)+D(u)Q_0)\colon \partial_N(a-\Delta)Q = \partial_N(Q_0\nabla^Tu+\nabla^TuQ_0)\colon \partial_N(a-\Delta)Q, $$
where we used the fact that $A\colon B=0$ if $A$ is anti-symmetric and $B$ is symmetric. Finally, distinguishing between the cases $k=N$ and $k<N$, we can use the same argument as in Lemma \ref{l.en-can.2} to conclude.
\end{proof}
\begin{lem}\label{l.en-can.4}
Let $\Omega=\mathbb R^N,\mathbb R^N_+$ with $N=2,3$, let $a,T>0$, $(u,Q)\in Y_T(\Omega)$  as in Lemma \ref{l.en-can.1}, let $Q_0\in H^2_{\Delta_N}(\Omega;S_0(N,\mathbb R))$, then for 
a.e. $t\in(0,T)$ it holds
\begin{equation}\label{en-can.4-0}
    \int_{\Omega}\mathbb P{\rm Div}\left(Q_0(\Delta-a)Q - (\Delta-a)Q Q_0\right)\cdot u+(W(u)Q_0 - Q_0W(u))\colon (a-\Delta)Qdx=0,
\end{equation}
\begin{equation}\label{en-can.4-i}
    \begin{aligned}
        & \left|\int_0^t\int_{\Omega}\mathbb P{\rm Div}\left(Q_0(\Delta-a)Q - (\Delta-a)Q Q_0\right)\cdot (-\partial_i^2)u+\partial_i(W(u)Q_0 - Q_0W(u))\colon \partial_i(a-\Delta)Qdxd\tau\right| \\
        & \lesssim t^\gamma \|Q_0\|_{H^2(\Omega)}\|Q\|_{X^2_t(\Omega)}\left(\|u\|_{L^\infty((0,t);H^1(\Omega))} + \|\nabla u\|_{X^0_t(\Omega)}\right)\quad i<N
    \end{aligned}
\end{equation}
\begin{equation}\label{en-can.4-N}
    \begin{aligned}
        & \left|\int_0^t\int_{\Omega}\mathbb P{\rm Div}\left(Q_0(\Delta-a)Q - (\Delta-a)Q Q_0\right)\cdot (-\partial_N^2)u+\partial_N(W(u)Q_0 - Q_0W(u))\colon \partial_N(a-\Delta)Qdxd\tau\right| \\
        & \lesssim t^\gamma \|Q_0\|_{H^2(\Omega)}\|Q\|_{X^2_t(\Omega)}\left(\|u\|_{L^\infty((0,t);H^1(\Omega))} + \|\nabla u\|_{X^0_t(\Omega)}\right) + \\
        & + \|Q_0\|_{H^2(\Omega)}\left(\|\nabla^\prime \nabla u\|_{L^2((0,t);L^2(\Omega))}\|Q\|_{X^2_t(\Omega)} + \|\nabla u\|_{X^0_t(\Omega)}\|\nabla^\prime Q\|_{L^2((0,t);H^2(\Omega))}\right),
    \end{aligned}
\end{equation}
for some $\gamma>0$.
\end{lem}
The proof is the same, using the second identity of Lemma \ref{l.matrix-trace} in place of the first one. As a consequence of Lemmas \ref{l.en-can.2}, \ref{l.en-can.3} and \ref{l.en-can.4}, we have the following result:
\begin{lem}\label{l.en-can.fg}
Let $\Omega=\mathbb R^N,\mathbb R^N_+$ with $N=2,3$, let $a,T>0$, let $(u,Q)\in Y_T(\Omega)$ as in Lemma \ref{l.en-can.1}, let $Q_0\in H^2_{\Delta_N}(\Omega)$, then for a.e. $t\in(0,T)$ it holds
\begin{equation}\label{en-can.fg-0}
    \int_{\Omega}\mathbb Pf(Q,Q_0,Q_0)\cdot u+g(u,Q_0,Q_0)\colon (a-\Delta)Qdx=0,
\end{equation}
\begin{equation}\label{en-can.fg-i}
    \begin{aligned}
        & \left|\int_0^t\int_{\Omega}\mathbb Pf(Q,Q_0,Q_0)\cdot (-\partial_i^2)u+\partial_ig(u,Q_0,Q_0)\colon \partial_i(a-\Delta)Qdxd\tau\right| \\
        & \lesssim t^\gamma \|Q\|_{X^2_t(\Omega)}\left(1+\|Q_0\|_{H^2(\Omega)}^2\right)\left(\|u\|_{L^\infty((0,t);H^1(\Omega))} + \|\nabla u\|_{X^0_t(\Omega)}\right)\quad i<N
    \end{aligned}
\end{equation}
\begin{equation}\label{en-can.fg-N}
    \begin{aligned}
        & \left|\int_0^t\int_{\Omega}\mathbb Pf(Q,Q_0,Q_0)\cdot (-\partial_N^2)u+\partial_Ng(u,Q_0,Q_0)\colon \partial_N(a-\Delta)Qdxd\tau\right| \\
        & \lesssim t^\gamma \|Q\|_{X^2_t(\Omega)}\left(1+\|Q_0\|_{H^2(\Omega)}^2\right)\left(\|u\|_{L^\infty((0,t);H^1(\Omega))} + \|\nabla u\|_{X^0_t(\Omega)}\right) + \\
        & + \left(1+\|Q_0\|_{H^2(\Omega)}^2\right)\left(\|\nabla^\prime \nabla u\|_{L^2((0,t);L^2(\Omega))}\|Q\|_{X^2_t(\Omega)} + \|\nabla u\|_{X^0_t(\Omega)}\|\nabla^\prime Q\|_{L^2((0,t);H^2(\Omega))}\right),
    \end{aligned}
\end{equation}
for some $\gamma>0$.
\end{lem}
\begin{rem}
We mentioned before that the case $\Omega=\mathbb R^N$ exhibits stronger cancellations, mainly due to the absence of the boundary. In this case, we can directly rely on the inequalities \eqref{en-can.fg-0} and \eqref{en-can.fg-i} for any $i\in\{1,\ldots, N\}$.   
\end{rem}
Lemma \ref{l.en-can.fg} shows the cancellations that appear in the nonlinear terms thanks to the energy method. These are fundamental to prove a uniform bound on the solutions $(u_\varepsilon,Q_\varepsilon)$ of Proposition \ref{p.loc.ex-approx.}. However, we also need a bound from below for the left hand side of the system \eqref{approx.lin.sys.2} when we use the energy method:
\begin{lem}\label{l.en.met-LHS1}
Let $\Omega=\mathbb R^N,\mathbb R^N_+$ with $N=2,3$, let $T>0$, $u\in X^1_T(\Omega)$ with $u(t)\in H^2_{\mathbb P\Delta_D}(\Omega;\mathbb R^N)$ for a.e. $t\in(0,T)$ and $u(0)=u_0\in H^1_{\mathbb P\Delta_D}(\Omega)$, then for a.e. $t\in(0,T)$ it holds
\begin{equation}\label{en-met.u-0}
    \int_0^t\int_{\Omega}(\partial_t-\mathbb P\Delta)u\cdot R_\varepsilon^Su dxd\tau  \gtrsim \|R_\varepsilon^Su(t)\|_{L^2(\Omega)}^2-\|u_0\|_{L^2(\Omega)}^2 + \|\nabla R_\varepsilon^Su\|_{L^2((0,t);L^2(\Omega))}^2,
\end{equation}
\begin{equation}\label{en-met.u-1i}
    \begin{aligned}
        & \int_0^t\int_{\Omega}(\partial_t-\mathbb P\Delta)u\cdot (-\partial^2_i)R_\varepsilon^Su dxd\tau \\ 
        & \gtrsim \|\partial_iR_\varepsilon^Su(t)\|_{L^2(\Omega)}^2-\|u_0\|_{H^1(\Omega)}^2 + \|\nabla \partial_iR_\varepsilon^Su\|_{L^2((0,t);L^2(\Omega))}^2\quad i=1,\ldots,N.
    \end{aligned}
\end{equation}
\end{lem}
\begin{proof}\hfill\\
\textbf{Step 1: Proof of \eqref{en-met.u-0}.} For simplicity, we denote $R_\varepsilon=R_\varepsilon^S$. As before, we can apply Remark \ref{rem.help-calc.} to remove the operator $\mathbb P$. We recall the resolvent identity from Lemma \ref{l.res.prop.}:
\begin{equation}\label{res.id.u}
    \begin{aligned}
        u=R_\varepsilon u-\varepsilon\Delta R_\varepsilon u.
    \end{aligned}
\end{equation}
So
$$ \int_{\mathbb R^N_+}(\partial_t-\Delta)u\cdot R_\varepsilon u dx  = \int_{\mathbb R^N_+}(\partial_t-\Delta)(1-\varepsilon\Delta)R_\varepsilon u\cdot R_\varepsilon u dx. $$
Thanks to Lemma \ref{l.res.prop.}, we know that $R_\varepsilon u=\Delta R_\varepsilon  u=0$ on $\mathbb R^N_0$, so it holds
$$ \int_{\mathbb R^N_+}(\partial_t-\Delta)(1-\varepsilon\Delta)R_\varepsilon u\cdot R_\varepsilon u dx $$
$$ = \frac{1}{2}\frac{d}{dt}\left[\|R_\varepsilon u\|_{L^2(\mathbb R^N_+)}^2+\varepsilon\|\nabla R_\varepsilon u\|_{L^2(\mathbb R^N_+)}^2\right] + \|\nabla R_\varepsilon u\|_{L^2(\mathbb R^N_+)}^2 + \varepsilon\|\Delta R_\varepsilon u\|_{L^2(\mathbb R^N_+)}^2 . $$
So, if we take the integral over $\tau\in(0,t)$ we get
$$ \int_0^t\int_{\mathbb R^N_+}(\partial_t-\Delta)(1-\varepsilon\Delta)R_\varepsilon u\cdot R_\varepsilon u dxd\tau  $$
$$ = \frac{1}{2}\left[\|R_\varepsilon u(t)\|_{L^2(\mathbb R^N_+)}^2+\varepsilon\|\nabla R_\varepsilon u(t)\|_{L^2(\mathbb R^N_+)}^2 - \|R_\varepsilon u_0\|_{L^2(\mathbb R^N_+)}^2-\varepsilon\|\nabla R_\varepsilon u_0\|_{L^2(\mathbb R^N_+)}^2\right] + $$
$$ + \|\nabla R_\varepsilon u\|_{L^2_tL^2}^2 + \varepsilon\|\Delta R_\varepsilon u\|_{L^2_tL^2}^2 $$
$$ \gtrsim \|R_\varepsilon u(t)\|_{L^2(\mathbb R^N_+)}^2 - \|u_0\|_{L^2(\mathbb R^N_+)}^2 + \|\nabla R_\varepsilon u\|_{L^2_tL^2}^2, $$
where, thanks to Lemma \ref{l.res.es.}, we could apply the inequality
$$ \|R_\varepsilon u_0\|_{L^2(\mathbb R^N_+)}^2+\varepsilon\|\nabla R_\varepsilon u_0\|_{L^2(\mathbb R^N_+)}^2\lesssim \|u_0\|_{L^2(\mathbb R^N_+)}^2. $$

\bigskip
\noindent \textbf{Step 2: Proof of \eqref{en-met.u-1i}.} Thanks to \eqref{res.id.u}
$$ \int_{\mathbb R^N_+}(\partial_t-\Delta)u\cdot (-\partial_i^2)R_\varepsilon u dx  = \int_{\mathbb R^N_+}(\partial_t-\Delta)(1-\varepsilon\Delta)R_\varepsilon u\cdot (-\partial_i^2)R_\varepsilon u dx. $$
We focus on the term
$$ \int_{\mathbb R^N_+}(-\Delta)^2R_\varepsilon u\cdot (-\partial_i^2)R_\varepsilon u dx= \int_{\mathbb R^N_+}\nabla (-\Delta)R_\varepsilon u\colon\nabla (-\partial_i^2)R_\varepsilon u dx, $$
where we used that $\partial_iR_\varepsilon u=0$ on $\mathbb R^N_0$ thanks to Remark \ref{rem.trace-R_eps.}. When $i<N$, it is easy to see that $\partial_i R_\varepsilon u(\tau)=R_\varepsilon \partial_i u(\tau)\in H^2(\Omega)$ for a.e. $\tau\in(0,T)$. So
$$ \int_{\mathbb R^N_+}\nabla (-\Delta) R_\varepsilon u\colon \nabla (-\partial_i^2)R_\varepsilon u dx = \int_{\mathbb R^N_+}\nabla \Delta R_\varepsilon u\colon \nabla\partial_iR_\varepsilon \partial_iu dx =  \int_{\mathbb R^N_+}\nabla \Delta R_\varepsilon u\colon \partial_i\nabla R_\varepsilon \partial_iu dx $$
$$ = - \int_{\mathbb R^N_+}\nabla \partial_i \Delta R_\varepsilon u\colon \nabla R_\varepsilon \partial_iu dx =\int_{\mathbb R^N_+} \partial_i \Delta R_\varepsilon u\cdot \Delta R_\varepsilon \partial_iu dx = \|\partial_i\Delta R_\varepsilon u\|_{L^2(\mathbb R^N_+)}^2, $$
where we used that $\partial_i\Delta R_\varepsilon u=0$ on $\mathbb R^N_0$ and the identity
$$ \Delta R_\varepsilon \partial_i u=\frac{1}{\varepsilon}\left(R_\varepsilon\partial_i u-\partial_i u\right)=\frac{\partial_i}{\varepsilon}\left(R_\varepsilon u-u\right)=\partial_i\Delta R_\varepsilon u. $$
When $i=N$
$$ \int_{\mathbb R^N_+}\nabla (-\Delta) R_\varepsilon u\colon \nabla (-\partial_N^2)R_\varepsilon u dx $$
$$ = \|\nabla \Delta R_\varepsilon u\|_{L^2(\mathbb R^N_+)}^2 - \sum_{i=1}^{N-1}\int_{\mathbb R^N_+}\nabla (-\Delta) R_\varepsilon u\colon \nabla (-\partial_i^2)R_\varepsilon u dx  $$
$$ = \|\nabla \Delta R_\varepsilon u\|_{L^2(\mathbb R^N_+)}^2 - \sum_{i=1}^{N-1} \|\partial_i\Delta R_\varepsilon u\|_{L^2(\mathbb R^N_+)}^2 = \|\partial_N\Delta R_\varepsilon u\|_{L^2(\mathbb R^N_+)}^2. $$
Therefore
$$  \int_{\mathbb R^N_+}(\partial_t-\Delta)(1-\varepsilon\Delta)R_\varepsilon u\cdot (-\partial_i^2)R_\varepsilon u dx $$
$$ = \frac{d}{dt}\left[\|\partial_iR_\varepsilon u\|_{L^2(\mathbb R^N_+)}^2+\varepsilon\|\partial_i\nabla R_\varepsilon u\|_{L^2(\mathbb R^N_+)}^2\right] + \|\partial_i\nabla R_\varepsilon u\|_{L^2(\mathbb R^N_+)}^2 + \varepsilon\|\partial_i\Delta R_\varepsilon u\|_{L^2(\mathbb R^N_+)}^2. $$
Taking the integral on $(0,t)$ and applying Lemma \ref{l.res.es.} for
$$ \|\partial_iR_\varepsilon u_0\|_{L^2(\mathbb R^N_+)}^2+\varepsilon\|\partial_i\nabla R_\varepsilon u_0\|_{L^2(\mathbb R^N_+)}^2 \lesssim \|u_0\|_{H^1(\mathbb R^N_+)}^2 $$
we conclude as before. 
\end{proof}
\begin{lem}\label{l.en.met-LHS2}
Let $\Omega=\mathbb R^N,\mathbb R^N_+$, let $a,T>0$, $Q\in X^2_T(\Omega)$ with $Q(t)\in H^3_{\Delta_N}(\Omega;S_0(N,\mathbb R))$ for a.e. $t\in(0,T)$ and $Q(0)=Q_0\in H^2_{\Delta_N}(\Omega)$, then for a.e. $t\in(0,T)$ it holds
\begin{equation}\label{en-met.Q-0}
    \begin{aligned}
        & \int_0^t\int_{\Omega}(\partial_t+a-\Delta)Q\colon (a-\Delta)R_\varepsilon^HQ dxd\tau \\
        & \gtrsim \|R_\varepsilon^HQ(t)\|_{H^1(\Omega)}^2-\|Q_0\|_{H^1(\Omega)}^2 + \|R_\varepsilon^HQ\|_{L^2((0,t);H^2(\Omega))}^2,
    \end{aligned}
\end{equation}
\begin{equation}\label{en-met.Q-1i}
    \begin{aligned}
        & \int_0^t\int_{\Omega}\partial_i(\partial_t+a-\Delta)Q\colon \partial_i(a-\Delta)R_\varepsilon^HQ dxd\tau \\ 
        & \gtrsim \|\partial_iR_\varepsilon^HQ(t)\|_{H^1(\Omega)}^2-\|Q_0\|_{H^2(\Omega)}^2 + \| \partial_iR_\varepsilon^HQ\|_{L^2((0,t);H^2(\Omega))}^2\quad i=1,\ldots,N.
    \end{aligned}
\end{equation}
\end{lem}
The proof is similar to the previous one. Note that, unlike Lemma \ref{l.en.met-LHS1}, Lemma \ref{l.en.met-LHS2} also includes the $L^2_tL^2$-norm of $R_\varepsilon^H Q$ in the estimate. This arises from the presence of $a>0$.

\subsection{A uniform bound and the $\varepsilon$-limit}

In the previous section, we showed the cancellations obtained by applying the energy method. In particular, through the results from Lemma \ref{l.en-can.1} to Lemma \ref{l.en.met-LHS2}, we can prove an estimate on the solutions $(u_\varepsilon,Q_\varepsilon)$ from Proposition \ref{p.loc.ex-approx.} that does not depend on the parameter $\varepsilon$. Before we discuss the uniform bound, we need the following lemma:
\begin{lem}\label{l.unif.bound.}
Let $C_1,C_2,\alpha,T,\gamma>0$, let $x,y\ge0$ such that
$$ \left\{\begin{array}{l}
     x^2\le \alpha+ C_1T^\gamma y^2 \\
     y^2\le \alpha+ C_2\left(T^\gamma y^2 +xy\right) 
\end{array}\right. $$
then there is $T=T(C_1,C_2,\gamma)$ such that $|(x,y)|\le C\alpha$ for some $C>0$. 
\end{lem}
\begin{proof}\hfill\\
It is sufficient to note that
$$ xy\le \frac{C_2}{2} x^2 + \frac{1}{2C_2}y^2. $$
So 
$$ y^2\le \alpha + C_2T^\gamma y^2 + \frac{C^2_2}{2}x^2 + \frac{1}{2}y^2, $$
which implies
$$ y^2 \le 2\alpha + 2C_2T^\gamma y^2 + C_2^2x^2\le \alpha(2+C_2^2) + T^\gamma (2C_2+C_2^2C_1)y^2. $$
In particular, choosing
$$ T\le (4C_2+2C_2^2C_1)^{-1/\gamma}, $$
we get the thesis.
\end{proof}
We are now ready to prove the uniform bound:
\begin{prop}\label{p.res-unif.es.}
Let $\Omega=\mathbb R^N,\mathbb R^N_+$ with $N=2,3$, let $u_0\in H^1_{\mathbb P\Delta_D}(\Omega;\mathbb R^N)$ and $Q_0\in H^2_{\Delta_N}(\Omega;S_0(N,\mathbb R))$, then there is $T\in(0,1)$ such that, for any  
$$ \widetilde F\in L^2\left((0,T);L^2\left(\Omega;\mathbb R^N\right)\right),\quad \widetilde G\in L^2\left((0,T);H^1\left(\Omega;S_0(N,\mathbb R)\right)\right), $$
and for any $\varepsilon>0$ the solution $(u_\varepsilon,Q_\varepsilon)$ of the approximated linear system \eqref{approx.lin.sys.0} found in Proposition \ref{p.loc.ex-approx.} with $W=Q_0$ satisfies the estimate
\begin{equation}\label{Res-unif.YT.es.}
    \begin{aligned}
        & \left\|\left(R_\varepsilon^Su_\varepsilon,R_\varepsilon^HQ_\varepsilon\right)\right\|_{Y_{T}(\Omega)} \\
        \le C_1\Big[\|u_0\|_{H^1(\Omega)} + \|Q_0&\!\! \left.\|_{H^2(\Omega)} + \|\widetilde F\|_{L^2((0,T);L^2(\Omega))} + \|\widetilde G\|_{L^2((0,T);H^1(\Omega))}\right],
    \end{aligned}
\end{equation}
for some $C_1>0$ independent of $\varepsilon>0$. Moreover,
\begin{equation}\label{Res-unif.0.es.}
    \begin{aligned}
        & \left\|\left(R_\varepsilon^Su_\varepsilon,R_\varepsilon^HQ_\varepsilon\right)\right\|_{X^0_{T}(\Omega)\times X^1_{T}(\Omega)}\\
        \le C_2\Big[ \|u_0\|_{L^2(\Omega)} + \|Q_0&\|_{H^1(\Omega)} + \|\widetilde F\|_{L^1((0,T);L^2(\Omega))} + \|\widetilde G\|_{L^1((0,T);H^1(\Omega))}\Big],
    \end{aligned}
\end{equation}
for some $C_2>0$ independent of $\varepsilon>0$.
\end{prop}
\begin{proof}\hfill\\
\textbf{Step 1:} First, we prove the following inequality:
\begin{equation}\label{proof.en.in-0}
    \begin{aligned}
    & \|R_\varepsilon^Su_\varepsilon\|_{L^\infty_TL^2} + \|\nabla R_\varepsilon^Su_\varepsilon\|_{L^2_TL^2} + \|R_\varepsilon^HQ_\varepsilon\|_{X^1_T(\Omega)}     \\
    \lesssim   &  \|u_0\|_{L^2(\Omega)} + \|Q_0\|_{H^1(\Omega)} + \|\widetilde F\|_{L^1_TL^2} + \|\widetilde G\|_{L^1_TH^1}.
    \end{aligned}
\end{equation}
With respect to \eqref{Res-unif.0.es.}, the inequality \eqref{proof.en.in-0} does not cover the $L^2_TL^2$-norm of $R_\varepsilon^Su_\varepsilon$. In fact, to bound it, we need a different approach. For this reason, we treat the $L^2_TL^2$-norm of $R_\varepsilon^Su_\varepsilon$ separately later. To prove \eqref{proof.en.in-0}, we use the energy method: firstly we multiply the first equation of \eqref{approx.lin.sys.2} by $R_\varepsilon^Su_\varepsilon$ and we sum it to the second equation multiplied by $(a-\Delta)R_\varepsilon^HQ_\varepsilon$, that is
\begin{equation}\label{proof.en.met-0}
\begin{aligned}
    & \int_{\Omega}(\partial_t-\mathbb P\Delta)u_\varepsilon\cdot R_\varepsilon^Su_\varepsilon + (\partial_t+a-\Delta)Q_\varepsilon\colon (a-\Delta)R_\varepsilon^HQ_\varepsilon dx   \\
  =& -\beta\int_{\Omega} {\rm Div}(\Delta-a)R_\varepsilon^HQ_\varepsilon\cdot R_\varepsilon^Su_\varepsilon -  D(R_\varepsilon^Su_\varepsilon)\colon (a-\Delta)R_\varepsilon^HQ_\varepsilon dx + \\
  + \int_{\Omega}&\left(f(R_\varepsilon^HQ_\varepsilon,Q_0,Q_0) + \widetilde F\right)  \cdot R_\varepsilon^Su_\varepsilon + \left(g(R_\varepsilon^Su_\varepsilon,Q_0,Q_0) + \widetilde G\right)\colon (a-\Delta)R_\varepsilon^HQ_\varepsilon  dx.
\end{aligned}
\end{equation}
We integrate over $\tau\in(0,t)$ and we take the $\sup$ over $t\in(0,T)$. 
For what concerns the right hand side, by Lemma \ref{l.en-can.1} and \ref{l.en-can.fg} it holds
\begin{equation}\label{proof.es-RHS0}
\begin{aligned}
    \int_{\Omega}{\rm Div}(\Delta-a)R_\varepsilon^HQ_\varepsilon\cdot R_\varepsilon^Su_\varepsilon - D(R_\varepsilon^Su_\varepsilon)\colon (a-\Delta)R_\varepsilon^HQ_\varepsilon dx =0, \\
    \int_{\Omega}f(R_\varepsilon^HQ_\varepsilon,Q_0,Q_0)\cdot R_\varepsilon^Su_\varepsilon + g(R_\varepsilon^Su_\varepsilon,Q_0,Q_0)\colon (a-\Delta)R_\varepsilon^HQ_\varepsilon dx=0.
\end{aligned}
\end{equation}
Moreover
\begin{equation}\label{proof.es-FG0}
\begin{aligned}
    \left|\int_0^t\int_{\Omega} \widetilde F\cdot R_\varepsilon^Su_\varepsilon dxd\tau\right|\le \|\widetilde F\|_{L^1_TL^2}\|R_\varepsilon^Su_\varepsilon\|_{L^\infty_TL^2},  \\
 \left|\int_0^t\int_{\Omega} \widetilde G\colon (a-\Delta)R_\varepsilon^HQ_\varepsilon dxd\tau\right|\le  \|\widetilde G\|_{L^1_TH^1}\|R_\varepsilon^HQ_\varepsilon\|_{L^\infty_TH^1},
\end{aligned}    
\end{equation}
where we also used the boundary condition $\partial_N R_\varepsilon^HQ_\varepsilon=0$ on $\mathbb R^N_0$. For the left hand side of \eqref{proof.en.met-0}, we can apply the inequalities \eqref{en-met.u-0} and \eqref{en-met.Q-0}, so that 
\begin{equation}\label{proof.es-LHS0}
\begin{aligned}
    \sup_{t\in(0,T)}\int_0^t\int_{\Omega}(\partial_t-\mathbb P\Delta)u_\varepsilon\cdot R_\varepsilon^Su_\varepsilon + (\partial_t+a-\Delta)Q_\varepsilon\colon (a-\Delta)R_\varepsilon^HQ_\varepsilon dxd\tau \\
    \gtrsim \|(R_\varepsilon^Su_\varepsilon,R_\varepsilon^HQ_\varepsilon)\|_{L^\infty_T(L^2\times H^1)}^2 + \|\nabla R_\varepsilon^Su_\varepsilon\|_{L^2_TL^2}^2 + \|R_\varepsilon^HQ_\varepsilon\|_{L^2_TH^2}^2 -  \|u_0\|_{L^2(\Omega)}^2 - \|Q_0\|_{H^1(\Omega)}^2
\end{aligned}   
\end{equation}
So, combining \eqref{proof.es-RHS0}, \eqref{proof.es-FG0} and \eqref{proof.es-LHS0}, we get \eqref{proof.en.in-0}. 

\bigskip

\noindent \textbf{Step 2:} Similarly to the previous step, we prove the inequality \eqref{Res-unif.YT.es.} without the $L^2_TL^2$-norm of $R_\varepsilon^S u_\varepsilon$, that is
\begin{equation}\label{proof.en.in-1}
    \begin{aligned}
    & \|R_\varepsilon^Su_\varepsilon\|_{L^\infty_TH^1(\Omega)} + \|\nabla R_\varepsilon^Su_\varepsilon\|_{L^2_TH^1} + \|R_\varepsilon^HQ_\varepsilon\|_{X^2_T(\Omega)}     \\
    &\lesssim     \|u_0\|_{H^1(\Omega)} + \|Q_0\|_{H^2(\Omega)} + \|\widetilde F\|_{L^2_TL^2} + \|\widetilde G\|_{L^2_TH^1}.
    \end{aligned}
\end{equation}
Let us multiply the first equation of \eqref{approx.lin.sys.2} by $(-\partial_i^2)R_\varepsilon^Su_\varepsilon$ and let us sum it with the derivative $\partial_i$ of the second equation of \eqref{approx.lin.sys.2} multiplied by $\partial_i(a-\Delta)R_\varepsilon^HQ_\varepsilon$ for $i=1,\ldots,N$, that is: 
$$ \int_{\Omega}(\partial_t-\mathbb P\Delta)u_\varepsilon\cdot (-\partial_i^2)R_\varepsilon^Su_\varepsilon  +\partial_i(\partial_t+a-\Delta)Q_\varepsilon\colon \partial_i(a-\Delta)R_\varepsilon^HQ_\varepsilon dx  $$ 
$$  = -\beta\int_{\Omega}{\rm Div}(\Delta-a)R_\varepsilon^HQ_\varepsilon\cdot(-\partial_i^2) R_\varepsilon^Su_\varepsilon - \partial_iD(R_\varepsilon^Su_\varepsilon)\colon \partial_i(a-\Delta)R_\varepsilon^HQ_\varepsilon dx + $$
$$ + \int_{\Omega}f(R_\varepsilon^HQ_\varepsilon,Q_0,Q_0)\cdot (-\partial_i^2)R_\varepsilon^Su_\varepsilon + \partial_ig(R_\varepsilon^Su_\varepsilon,Q_0,Q_0)\colon \partial_i(a-\Delta)R_\varepsilon^HQ_\varepsilon dx + $$
$$ + \int_{\Omega} \widetilde F\cdot (-\partial_i^2)R_\varepsilon^Su_\varepsilon + \partial_i\widetilde G\colon \partial_i(a-\Delta)R_\varepsilon^HQ_\varepsilon dx. $$
The argument is the same as before: for the left hand side, we can apply \eqref{en-met.u-1i} and \eqref{en-met.Q-1i}. For what concerns the right hand side, first it holds
\begin{equation}\label{proof.es-FG1}
    \begin{aligned}
        \left|\int_0^t\int_{\Omega} \widetilde F\cdot (-\partial_i^2) R_\varepsilon^S u_\varepsilon dx d\tau\right|\le \|\widetilde F\|_{L^2_TL^2}\|\partial_i^2 R_\varepsilon^S u_\varepsilon\|_{L^2_TL^2}, \\
 \left|\int_0^t\int_{\Omega} \partial_i\widetilde G\colon \partial_i(a-\Delta)R_\varepsilon^HQ_\varepsilon dx d\tau\right|\le \|\widetilde G\|_{L^2_TH^1}\|\partial_i R_\varepsilon^HQ_\varepsilon\|_{L^2_TH^2}.
    \end{aligned}
\end{equation}
For the remaining terms we distinguish the case $i<N$ from the case $i=N$: when $i<N$ we apply \eqref{en-can.1-i} and \eqref{en-can.fg-i} which, combined with the previous inequalities, brings to the estimate \eqref{Res-unif.YT.es.} with the tangential derivatives, that is
\begin{equation}\label{proof.unif.es.tang.}
    \begin{aligned}
        & \|(R_\varepsilon^S u_\varepsilon,\nabla^\prime R_\varepsilon^Su_\varepsilon)\|_{L^\infty_TL^2}^2 + \|(R_\varepsilon^HQ_\varepsilon,\nabla^\prime R_\varepsilon^HQ_\varepsilon)\|_{L^\infty_TH^1}^2 + \\
        & + \|\nabla^\prime R_\varepsilon^Su_\varepsilon\|_{L^2_TH^1}^2 + \|(R_\varepsilon^HQ_\varepsilon,\nabla^\prime R_\varepsilon^HQ_\varepsilon)\|_{L^2_TH^2}^2 \\
     & \lesssim \|u_0\|_{H^1(\Omega)}^2 + \|Q_0\|_{H^2(\Omega)}^2 + \|\widetilde F\|_{L^2_TL^2}^2 + \|\widetilde G\|_{L^2_TH^1}^2 + \\
       & + C(Q_0)T^\gamma \|R_\varepsilon^HQ_\varepsilon\|_{X^2_T(\Omega)}\left(\|R_\varepsilon^Su_\varepsilon\|_{L^\infty_TH^1} + \|\nabla R^S_\varepsilon u_\varepsilon\|_{X^0_T(\Omega)}\right),
    \end{aligned}
\end{equation}
with
$$ C(Q_0)=1+\|Q_0\|_{H^2(\Omega)}^2. $$
Conversely, when $i=N$, we apply \eqref{en-can.1-N} and \eqref{en-can.fg-N} and we get the estimate for the transversal derivative:
\begin{equation}\label{proof.unif.es.trans.}
    \begin{aligned}
        & \|(R_\varepsilon^Su_\varepsilon,\partial_N R_\varepsilon^Su_\varepsilon)\|_{L^\infty_TL^2}^2 + \|(R_\varepsilon^HQ_\varepsilon,\partial_NR_\varepsilon^HQ_\varepsilon)\|_{L^\infty_TH^1}^2 + \\
        & +\|\partial_N R_\varepsilon^Su_\varepsilon\|_{L^2_TH^1}^2 + \|(R_\varepsilon^HQ_\varepsilon,\partial_N R_\varepsilon^HQ_\varepsilon)\|_{L^2_TH^2}^2 \\
     & \lesssim \|u_0\|_{H^1(\Omega)}^2 + \|Q_0\|_{H^2(\Omega)}^2 + \|\widetilde F\|_{L^2_TL^2}^2 + \|\widetilde G\|_{L^2_TH^1}^2 + \\
       & + C(Q_0)T^\gamma \|R_\varepsilon^HQ_\varepsilon\|_{X^2_T(\Omega)}\left(\|R_\varepsilon^Su_\varepsilon\|_{L^\infty_TH^1} + \|\nabla R^S_\varepsilon u_\varepsilon\|_{X^0_T(\Omega)}\right) + \\
       & + C(Q_0)\left[\|\nabla R^S_\varepsilon u_\varepsilon\|_{X^0_T(\Omega)}\|\nabla^\prime R^H_\varepsilon Q_\varepsilon\|_{L^2_TH^2} + \|\nabla^\prime R^S_\varepsilon u_\varepsilon\|_{L^2_TH^1}\|R^H_\varepsilon Q_\varepsilon\|_{X^2_T(\Omega)}\right].
    \end{aligned}
\end{equation}
Combining \eqref{proof.unif.es.tang.}, \eqref{proof.unif.es.trans.} and applying Lemma \ref{l.unif.bound.} with
$$ x=\|(R^S_\varepsilon u_\varepsilon,\nabla^\prime R^S_\varepsilon u_\varepsilon)\|_{L^\infty_TL^2} + \|(R^H_\varepsilon Q_\varepsilon,\nabla^\prime R^H_\varepsilon Q_\varepsilon)\|_{L^\infty_TH^1} + $$
$$ + \|\nabla^\prime R^S_\varepsilon u_\varepsilon\|_{L^2_TH^1} + \|(R^H_\varepsilon Q_\varepsilon,\nabla^\prime R^H_\varepsilon Q_\varepsilon)\|_{L^2_TH^2}, $$
$$ y = \|R_\varepsilon^Su_\varepsilon\|_{L^\infty_TH^1} + \|\nabla R^S_\varepsilon u_\varepsilon\|_{X^0_T(\Omega)} + \|R^H_\varepsilon Q_\varepsilon\|_{X^2_T(\Omega)}, $$
we can find $T>0$ independent of $\varepsilon$ sufficiently small such that \eqref{proof.en.in-1} holds.

\bigskip
\noindent \textbf{Step 3: The $L^2_TL^2$-norm of $R_\varepsilon^Su_\varepsilon$.} We recall that
\begin{equation}\label{proof.Res-in.}
    \|R_\varepsilon^Su_\varepsilon\|_{L^2_TL^2}\lesssim \|u_\varepsilon\|_{L^2_TL^2}.
\end{equation}
We already know that $(u_\varepsilon,Q_\varepsilon)$ solves the system \eqref{approx.lin.sys.2} in $(0,T)$. So, by the Duhamel Formula
$$ (u_\varepsilon,Q_\varepsilon)(t)= e^{Bt}(u_0,Q_0) + \int_0^t e^{B(t-\tau)}(f(R_\varepsilon^HQ_\varepsilon,Q_0,Q_0)+\widetilde F,g(R_\varepsilon^Su_\varepsilon,Q_0,Q_0)+\widetilde G)d\tau, $$
for $t\in(0,T)$, where $B$ is the operator associated with the $Q$-tensor system defined in \eqref{def.op.B}. Then, by Theorem \ref{t.sem.es.}, choosing $T<1$, it holds
\begin{equation}\label{proof.L2-es-int.}
\begin{aligned}
    & \|(u_\varepsilon,Q_\varepsilon)\|_{L^2(\Omega)\times H^1(\Omega)}\lesssim \|u_0\|_{L^2(\Omega)} + \|Q_0\|_{H^1(\Omega)} + \\
 + \int_0^t \|f(R_\varepsilon^HQ_\varepsilon,& Q_0,Q_0)(\tau) + \widetilde F(\tau)\|_{L^2(\Omega)} + \|g(R_\varepsilon^Su_\varepsilon,Q_0,Q_0)(\tau) + \widetilde G(\tau)\|_{H^1(\Omega)} d\tau.    
\end{aligned}
\end{equation}
From Lemma \ref{l.mult.loc.es.0} we know that  
$$ \|f(R_\varepsilon^HQ_\varepsilon,Q_0,Q_0) + \widetilde F\|_{L^1_TL^2} + \|g(R_\varepsilon^Su_\varepsilon,Q_0,Q_0) + \widetilde G\|_{L^1_TH^1} $$
$$ \lesssim T^{1/2}\left[\|(\widetilde F,\widetilde G)\|_{L^2_T(L^2\times H^1)} + \left(1 + \|Q_0\|_{H^2(\Omega)}^2\right)\|(\nabla R_\varepsilon^Su_\varepsilon,R_\varepsilon^HQ_\varepsilon)\|_{L^2_T(H^1\times H^3)}\right]. $$
So, by \eqref{proof.Res-in.}, \eqref{proof.en.in-1} and \eqref{proof.L2-es-int.}, there is $T\le1$ sufficiently small such that 
\begin{equation}\label{proof.L2-es.}
    \begin{aligned}
         \|(R_\varepsilon^Su_\varepsilon,R_\varepsilon^HQ_\varepsilon)\|_{L^2_T(L^2\times H^1)} \lesssim \|u_0\|_{H^1(\Omega)} + \|Q_0\|_{H^2(\Omega)} + \|\widetilde F\|_{L^2_TL^2} + \|\widetilde G\|_{L^2_TH^1}.
    \end{aligned}
\end{equation}
Finally, combining \eqref{proof.L2-es.} with \eqref{proof.en.in-0} and \eqref{proof.en.in-1} we conclude.
\end{proof}
Once we have the uniform bound for $R_\varepsilon^S u_\varepsilon$ and $R_\varepsilon^H Q_\varepsilon$, it is straightforward to obtain the one for $u_\varepsilon$ and $Q_\varepsilon$:
\begin{cor}\label{c.unif-es.}
Let $\Omega=\mathbb R^N,\mathbb R^N_+$ with $N=2,3$, let $u_0\in H^1_{\mathbb P\Delta_D}(\Omega;\mathbb R^N)$ and $Q_0\in H^2_{\Delta_N}(\Omega;S_0(N,\mathbb R))$, then there is $T>0$ such that, for any $\varepsilon>0$ and for any
$$ \widetilde F\in L^2\left((0,T);L^2\left(\Omega;\mathbb R^N\right)\right),\quad \widetilde G\in L^2\left((0,T);H^1\left(\Omega;S_0(N,\mathbb R)\right)\right), $$
the solution for the approximated linear system \eqref{approx.lin.sys.2} found in Proposition \ref{p.loc.ex-approx.} is uniformly bounded in $\varepsilon$, that is
\begin{equation}\label{unif.YT.es.}
    \|(u_\varepsilon,Q_\varepsilon)\|_{Y_{T}(\Omega)}\le C(Q_0)\left[ \|u_0\|_{H^1(\Omega)} + \|Q_0 \|_{H^2(\Omega)} + \|\widetilde F\|_{L^2((0,T);L^2(\Omega))} + \|\widetilde G\|_{L^2((0,T);H^1(\Omega))}\right],
\end{equation}
where
$$ C(Q_0)\sim 1+\|Q_0\|_{H^2(\Omega)}^2 $$
is independent of $\varepsilon>0$.
\end{cor}
\begin{proof}\hfill\\
Since $(u_\varepsilon,Q_\varepsilon)$ solves the approximated system \eqref{approx.lin.sys.2}, we can apply Theorem \ref{t.EL-lin.es.} and Lemma \ref{l.mult.loc.es.0} to obtain
$$ \|(u_\varepsilon,Q_\varepsilon)\|_{Y_T(\Omega)} \lesssim \|u_0\|_{H^1(\Omega)} + \|Q_0\|_{H^2(\Omega)} + \|\widetilde F\|_{L^2_TL^2} + \|\widetilde G\|_{L^2_TH^1} + $$
$$ + \|f(R_\varepsilon^H Q_\varepsilon, Q_0,Q_0)\|_{L^2_TL^2} + \|g(R_\varepsilon^Su_\varepsilon,Q_0,Q_0)\|_{L^2_TH^1} $$
$$ + \beta\|{\rm Div}(\Delta-a)R_\varepsilon^HQ_\varepsilon\|_{L^2_TL^2} + \beta\|D(R_\varepsilon^S u_\varepsilon)\|_{L^2_TH^1}$$
$$ \lesssim \|u_0\|_{H^1(\Omega)} + \|Q_0\|_{H^2(\Omega)} + \|\widetilde F\|_{L^2_TL^2} + \|\widetilde G\|_{L^2_TH^1} + $$
$$ + \left(1+\|Q_0\|_{H^2(\Omega)}^2\right)\|(R_\varepsilon^H Q_\varepsilon, R_\varepsilon^Su_\varepsilon)\|_{Y_T(\Omega)}. $$
So by Proposition \ref{p.res-unif.es.} we conclude.
\end{proof}
\noindent The previous two results give us the uniform bound on $(u_\varepsilon,Q_\varepsilon)$ and on $(R_\varepsilon^Su_\varepsilon,R^H_\varepsilon Q_\varepsilon)$. This is fundamental in order to pass to the limit on $\varepsilon\to0^+$ and to find a solution for the system
\begin{equation}\label{BE.i-lin.sys.1}
    \left\{\begin{array}{ll}
       (\partial_t-\mathbb P\Delta)u  + \beta \mathbb P{\rm Div}(\Delta-a)Q=\mathbb Pf(Q,Q_0,Q_0) + \mathbb P\widetilde F & (0,T)\times \Omega \\
       {\rm div}u=0 & (0,T)\times\Omega \\
       (\partial_t-\Delta+a)Q-\beta D(u)=g(u,Q_0,Q_0) + \widetilde G & (0,T)\times \Omega \\
       u=0,\quad \partial_\nu Q=0 & (0,T)\times\partial\Omega \\
       u(0)=u_0,\quad Q(0)=Q_0 & \Omega.
    \end{array}\right.
\end{equation}
\begin{prop}\label{p.ex.i-lin.sys.}
Let $\Omega=\mathbb R^N,\mathbb R^N_+$ with $N=2,3$, let $u_0\in H^1_{\mathbb P\Delta_D}(\Omega;\mathbb R^N)$ and $Q_0\in H^2_{\Delta_N}(\Omega;S_0(N,\mathbb R))$, then there is $T>0$ such that the linear system \eqref{BE.i-lin.sys.1} admits a solution $(u,Q)\in Y_T(\Omega)$ with $u(t)\in H^{2}_{\mathbb P\Delta_D}(\Omega;\mathbb R^N)$ and $Q(t)\in H^3_{\Delta_N}(\Omega;S_0(N,\mathbb R))$ for a.e. $t\in(0,T)$ with 
\begin{equation}\label{lim.un.es.YT}
\begin{aligned}
    &\|(u,Q)\|_{Y_T(\Omega)}\\
    \le C(Q_0)\Big[ \|u_0\|_{H^1(\Omega)} + \|Q_0\|_{H^2(\Omega)} + &\|\widetilde F\|_{L^2((0,T);L^2(\Omega))} + \|\widetilde G\|_{L^2((0,T);H^1(\Omega))}\Big],
\end{aligned}
\end{equation}
\begin{equation}\label{lim.un.es.0}
\begin{aligned}
    & \|(u,Q)\|_{X^0_{T}(\Omega)\times X^1_{T}(\Omega)} \\
  \le C(Q_0)\left[\|u_0\|_{L^2(\Omega)} + \|Q_0\|_{H^1(\Omega)} + \|\widetilde F\right.&\left.\|_{L^1((0,T);L^2(\Omega))} + \|\widetilde G\|_{L^1((0,T);H^1(\Omega))}\right],
\end{aligned}    
\end{equation}
where $C(Q_0)$ is the same of Corollary \ref{c.unif-es.}. Moreover, for any $s\in[0,1]$ and for any $\delta>0$, we can find $T>0$ such that
\begin{equation}\label{lim.un.es.diff.}
    \begin{aligned}
        & \|(u,Q)(t)-(u_0,Q_0)\|_{H^s(\Omega)\times H^{s+1}(\Omega)} \\
        \lesssim \delta + T^\frac{1-s}{2}&C(Q_0)\left[\|\widetilde F\|_{L^2((0,T);L^2(\Omega))} +\|\widetilde G\|_{L^2((0,T);H^1(\Omega))}\right]
    \end{aligned} 
\end{equation}
for a.e. $t\in(0,T)$.
\end{prop}
\begin{proof}\hfill\\
\textbf{Step 1: Proof of \eqref{lim.un.es.YT} and \eqref{lim.un.es.0}.} By Proposition \ref{p.res-unif.es.} and Corollary \ref{c.unif-es.}, we know that there is $T>0$ such that, for any $\varepsilon>0$, the system \eqref{approx.lin.sys.2} admits a solution $(u_\varepsilon,Q_\varepsilon)\in Y_T(\Omega)$ with 
$$ \|(R_\varepsilon^Su_\varepsilon,R_\varepsilon^HQ_\varepsilon)\|_{Y_T(\Omega)} \le C(Q_0)\left[\|u_0\|_{H^1(\Omega)} + \|Q_0\|_{H^2(\Omega)} + \|\widetilde F\|_{L^2_TL^2} + \|\widetilde G\|_{L^2_TH^1}\right], $$
$$ \|(u_\varepsilon,Q_\varepsilon)\|_{Y_T(\Omega)} \le C(Q_0)\left[ \|u_0\|_{H^1(\Omega)} + \|Q_0\|_{H^2(\Omega)} + \|\widetilde F\|_{L^2_TL^2} + \|\widetilde G\|_{L^2_TH^1}\right]. $$
Let $\varepsilon_n>0$ such that $\varepsilon_n\searrow0$ as $n\to+\infty$. For simplicity, we denote
$$ (u_n,Q_n)=(u_{\varepsilon_n},Q_{\varepsilon_n}), \quad R_n^A=R_{\varepsilon_n}^A \quad \text{for}\:\:A=S,H. $$
Then, 
$$ (R_n^Su_n,R_n^HQ_n)\rightharpoonup (u,Q), $$
$$ (u_n,Q_n)\rightharpoonup (v,W) $$
in $L^2_T(H^2\times H^3)$, and
$$ (\partial_t u_n,\partial_t Q_n)\rightharpoonup(\partial_t v,\partial_t W), $$
in $L^2_T(L^2\times H^1)$. On the other hand, thanks to the self-adjointness of $R_n^S$ and $R_n^H$, it is easy to see that $(u,Q)=(v,W)$. We can also assume that 
$$ (u_n,Q_n)(t)\xrightharpoonup{H^1(\Omega)\times H^2(\Omega)} (u,Q)(t) \quad \text{for a.e.}\:\:t\in(0,T). $$
Since $Q_0\in H^2(\mathbb R^N_+)$ and thanks to the weak convergence of $(R_n^Su_n,R_n^HQ_n)$, it can be seen that 
$$ f(R_n^HQ_n,Q_0,Q_0)\rightharpoonup f(Q,Q_0,Q_0)\quad \text{in}\:\:L^2_TL^2, $$
$$ g(R_n^Su_n,Q_0,Q_0)\rightharpoonup g(u,Q_0,Q_0)\quad \text{in}\:\:L^2_TH^1. $$
As a consequence, it can be seen that $(u,Q)$ solves the linear system \eqref{BE.i-lin.sys.1} in $L^2_T(L^2\times L^2)$. Moreover, let $\alpha=0,1$, then by lower semi-continuity of the norms it holds
$$ \|(u,Q)\|_{L^2_T(H^{\alpha+1}\times H^{\alpha+2})}\le \liminf_{n\to+\infty} \|(u_n,Q_n)\|_{L^2_T(H^{\alpha+1}\times H^{\alpha+2})} $$
$$ \le C(Q_0)\left[\|(u_0,Q_0)\|_{H^\alpha(\Omega)\times H^{\alpha+1}(\Omega)} + \|(\widetilde F,\widetilde G)\|_{L^{\alpha+1}_T(L^2\times H^1)}\right], $$
$$ \|(u,Q)\|_{H^1_T(L^2\times H^1)}\le \liminf_{n\to+\infty} \|(u_n,Q_n)\|_{H^1_T(L^2\times H^1)} $$
$$ \le C(Q_0)\left[\|(u_0,Q_0)\|_{H^1(\Omega)\times H^2(\Omega)} + \|(\widetilde F,\widetilde G)\|_{L^2_T(L^2\times H^1)}\right], $$
and for a.e. $t\in(0,T)$
$$ \|(u,Q)(t)\|_{H^\alpha(\Omega)\times H^{\alpha+1}(\Omega)} $$
$$ \le \liminf_{n\to\infty}\|(u_{n},Q_{n})(t)\|_{H^\alpha(\Omega)\times H^{\alpha+1}(\Omega)}\le C(Q_0)\left[\|(u_0,Q_0)\|_{H^\alpha(\Omega)\times H^{\alpha+1}(\Omega)} + \|(\widetilde F,\widetilde G)\|_{L^{\alpha+1}_T(L^2\times H^1)}\right]. $$
In particular, the last inequality implies 
$$ \|(u,Q)\|_{L^\infty_T(H^\alpha\times H^{\alpha+1})}\le C(Q_0)\left[\|(u_0,Q_0)\|_{H^\alpha(\Omega)\times H^{\alpha+1}(\Omega)} + \|(\widetilde F,\widetilde G)\|_{L^2_T(L^2\times H^1)}\right], $$
which concludes the proof of \eqref{lim.un.es.YT} and \eqref{lim.un.es.0}. 

\bigskip

\noindent \textbf{Step 2: Proof of \eqref{lim.un.es.diff.}.} The function $w=(u,Q)-e^{Bt}(u_0,Q_0)$ solves the system
$$ \left\{\begin{array}{ll}
    (\partial_t-B)w= (f(Q,Q_0,Q_0) + \widetilde F,g(u,Q_0,Q_0) + \widetilde G) & (0,T)\times\Omega \\
    w(0)=(0,0) & \Omega,
\end{array}\right. $$
where $B=(B_1,B_2)$ is defined in \eqref{def.op.B}. Let us denote
$$ S(t)=(S_1(t),S_2(t))=e^{Bt}(u_0,Q_0)$$
We can write
$$ f(Q,Q_0,Q_0)= f(Q-S_2(t),Q_0,Q_0) + f(S_2(t),Q_0,Q_0),  $$
$$ g(u,Q_0,Q_0)= g(u-S_1(t),Q_0,Q_0) + g(S_1(t),Q_0,Q_0). $$
So 
$$ \left\{\begin{array}{ll}
    (\partial_t-B_1)w_1= f(w_2,Q_0,Q_0) + f(S_2(t),Q_0,Q_0) + \widetilde F & (0,T)\times \Omega \\
    (\partial_t-B_2)w_2= g(w_1,Q_0,Q_0) +  g(S_1(t),Q_0,Q_0) + \widetilde G & (0,T)\times\Omega \\
    w_1=0,\quad \partial_\nu w_2=0 & (0,T)\times\partial\Omega \\
    w(0)=(0,0) & \Omega.
\end{array}\right. $$
By \eqref{lim.un.es.0} and \eqref{lim.un.es.YT}
$$ \|w\|_{X^0_T(\Omega)\times X^1_T(\Omega)} $$
$$ \le C(Q_0)\left[\|\widetilde F + f(S_2(t),Q_0,Q_0)\|_{L^1_TL^2} + \|\widetilde G + g(S_1(t),Q_0,Q_0))\|_{L^1_TH^1}\right], $$
$$ \|w\|_{Y_T(\Omega)}\le C(Q_0)\left[\|\widetilde F + f(S_2(t),Q_0,Q_0)\|_{L^2_TL^2} + \|\widetilde G + g(S_1(t),Q_0,Q_0))\|_{L^2_TH^1}\right]. $$
So, interpolating the two inequalities we obtain
$$ \|w\|_{X^s_T(\Omega)\times X^{s+1}_T(\Omega)} $$
$$ \le C(Q_0)\left[\|\widetilde F + f(S_2(t),Q_0,Q_0)\|_{L^\frac{2}{2-s}_TL^2} + \|\widetilde G + g(S_1(t),Q_0,Q_0))\|_{L^\frac{2}{2-s}_TH^1}\right]  $$
$$ \le T^\frac{1-s}{2}C(Q_0)\left[\|\widetilde F + f(S_2(t),Q_0,Q_0)\|_{L^2_TL^2} + \|\widetilde G + g(S_1(t),Q_0,Q_0))\|_{L^2_TH^1}\right]  $$
for $s\in[0,1]$. On the other hand, by Lemma \ref{l.mult.loc.es.0} and Theorem \ref{t.sem.es.} it holds
$$ \|f(S_2(t),Q_0,Q_0)\|_{L^2_TL^2} + \|g(S_1(t),Q_0,Q_0))\|_{L^2_TH^1} $$
$$ \le C(Q_0)e^{\alpha T}\left(\|u_0\|_{H^1(\Omega)} + \|Q_0\|_{H^2(\Omega)}\right), $$
for some $\alpha>0$.
Let now $\delta>0$, then we can find $T>0$ sufficiently small such that 
$$ C(Q_0)^2T^\frac{1-s}{2}e^{\alpha T}\left(\|u_0\|_{H^1(\Omega)} + \|Q_0\|_{H^2(\Omega)}\right)\le \frac{\delta}{2}. $$
Finally, 
$$ \|(u,Q)(t)-(u_0,Q_0)\|_{H^s(\Omega)\times H^{s+1}(\Omega)} $$
$$ \le \|(u,Q)(t)-e^{Bt}(u_0,Q_0)\|_{H^s(\Omega)\times H^{s+1}(\Omega)} + \|e^{Bt}(u_0,Q_0) - (u_0,Q_0)\|_{H^s(\Omega)\times H^{s+1}(\Omega)}.  $$
We have just bounded the first term, so we focus on the second one. As mentioned in Theorem \ref{t.sem.es.}, the semigroup $e^{Bt}$ is $C_0$-analytic. Therefore, by Lemma \ref{l.res.eq.B}
$$ \|(u_0,Q_0)-e^{Bt}(u_0,Q_0)\|_{H^s(\Omega)\times H^{s+1}(\Omega)} $$
$$ \lesssim \|(1-B)^{s/2}(u_0,Q_0)-(1-B)^{s/2}e^{Bt}(u_0,Q_0)\|_{L^2(\Omega)\times H^1(\Omega)} $$
$$ =  \|(1-B)^{s/2}(u_0,Q_0)-e^{Bt}(1-B)^{s/2}(u_0,Q_0)\|_{L^2(\Omega)\times H^1(\Omega)}\xrightarrow{t\to0}0, $$
and for $T>0$ sufficiently small
$$ \|(u_0,Q_0)-e^{Bt}(u_0,Q_0)\|_{H^s(\Omega)\times H^{s+1}(\Omega)}\le \frac{\delta}{2}. $$
\end{proof}

\subsection{Proof of Theorem \ref{t.loc.ex.}}

Now that we have proved the estimate for the system \eqref{BE.i-lin.sys.1}, we can focus on the local well-posedness for the general $Q$-tensor model. We have already obtained some estimates for the functions $f(u,Q)$ and $g(u,Q)$. For $\widetilde F(u,Q)$ and $\widetilde G(u,Q)$ it holds the following:
\begin{lem}\label{l.mult.loc.es.1}
Let $\Omega=\mathbb R^N,\mathbb R^N_+$ with $N=2,3$, let $T>0$ and $k\in\mathbb N$, let $v_1,v_2\in X^1_T(\Omega)$ and $w_j\in X^2_T(\Omega)$ for $j=1,\ldots,k$, then it holds 
$$ \|v_1\nabla v_2\|_{L^2((0,T);L^2(\Omega))}\lesssim T^{1/4}\|v_1\|_{X^1_T(\Omega)}\|v_2\|_{X^1_T(\Omega)}, $$
$$ \|v_1\nabla v_2 w_1\cdots w_k\|_{L^2((0,T);L^2(\Omega))}\lesssim T^{1/4} \|v_1\|_{X^1_T(\Omega)}\|v_2\|_{X^1_T(\Omega)}\prod_{j=1}^k \|w_j\|_{X^2_T(\Omega)}, $$
$$ \|v_1w_1\cdots w_k\|_{L^2((0,T);L^2(\Omega))}\lesssim T^{1/2} \|v_1\|_{X^1_T(\Omega)}\prod_{j=1}^k\|w_j\|_{X^2_T(\Omega)}, $$
$$ \|w_1\cdots w_k\|_{L^2((0,T);L^2(\Omega))}\lesssim T^{1/2} \prod_{j=1}^k\|w_j\|_{X^2_T(\Omega)}. $$
In particular, for any $v_1,v_2\in X^1_T(\Omega)$ and $w_1,w_2\in X^2_T(\Omega)$ and for any $j=1,2$
$$ \|\widetilde F(v_j,w_j)\|_{L^2((0,T);L^2(\Omega))} + \|\widetilde G(v_j,w_j)\|_{L^2((0,T);H^1(\Omega))}\lesssim T^{1/4}\left(1+\|v_j\|_{X^1_T(\Omega)}^5 + \|w_j\|_{X^2_T(\Omega)}^5\right),  $$
$$ \|\widetilde F(v_1,w_1) - \widetilde F(v_2,w_2)\|_{L^2((0,T);L^2(\Omega))} + \|\widetilde G(v_1,w_1) - \widetilde G(v_2,w_2)\|_{L^2((0,T);H^1(\Omega))} $$
$$ \lesssim T^{1/4}\left(1+\|v_j\|_{X^1_T(\Omega)}^4 + \|w_j\|_{X^2_T(\Omega)}^4\right)\left[\|v_1-v_2\|_{X^1_T(\Omega)} + \|w_1-w_2\|_{X^2_T(\Omega)}\right]. $$
\end{lem}
To prove these estimates, it is sufficient to apply H\"older and Sobolev inequalities. We are finally ready to prove Theorem \ref{t.loc.ex.}:
\begin{proof}[Proof of Theorem \ref{t.loc.ex.}]\hfill\\
We consider the map $\Phi\colon (z,V)\mapsto (u,Q)$, where $(u,Q)$ solves the system
$$ \left\{\begin{array}{ll}
    (\partial_t-\mathbb P\Delta)u +\beta\mathbb P{\rm Div}(\Delta-a)Q - \mathbb Pf(Q,Q_0,Q_0) = \mathbb P\widetilde f(V,Q_0) + \mathbb P\widetilde F(z,V)  & (0,T)\times\Omega \\
    {\rm div}u=0 & (0,T)\times\Omega \\
    (\partial_t+a-\Delta)Q - \beta D(u) -  g(u,Q_0,Q_0)= \widetilde g(z,V,Q_0) + \widetilde G(z,V) & (0,T)\times \Omega \\
    u=0, \quad \partial_\nu Q=0 & (0,T)\times\partial\Omega \\
    u(0)=u_0,\quad Q(0)=Q_0 & \Omega,
\end{array}\right. $$
where
$$ \widetilde f(V,Q_0)={\rm Div}\left[2\xi(\Delta-a)V\colon(V-Q_0)\left(V+\frac{Id}{N}\right) + 2\xi(\Delta-a)V\colon Q_0(V-Q_0)\right] - $$
$$ - {\rm Div}\left[(\xi+1)(\Delta-a)V(V-Q_0) + (1-\xi)(V-Q_0)(\Delta-a)V\right], $$
$$ \widetilde g(z,V,Q_0) = \xi\left[D(z)(V-Q_0) + (V-Q_0)D(z)\right] + W(z)(V-Q_0) - (V-Q_0)W(z)- $$
$$ - 2\xi(V-Q_0)V\colon \nabla z - 2\xi\left(Q_0+\frac{Id}{N}\right)(V-Q_0)\colon\nabla z. $$
It can be seen that 
$$ f(Q,Q_0,Q_0)+\widetilde f(Q,Q_0)=f(Q,Q,Q), $$
$$ g(u,Q_0,Q_0) + \widetilde g(u,Q,Q_0)=g(u,Q,Q), $$
So, if we find a fixed point for the map $\Phi$, we have a solution for the $Q$-tensor system. We know from Proposition \ref{p.ex.i-lin.sys.} that such $(u,Q)$ exists if
$$ \widetilde f(V,Q_0) + \widetilde{F}(z,V)\in L^2((0,T);L^2(\Omega;\mathbb R^N)), $$
$$ \widetilde g(z,V,Q_0) + \widetilde{G}(z,V)\in L^2((0,T);H^1(\Omega;\mathbb R^N)). $$
Following the same approach of Lemma \ref{l.mult.loc.es.0} and thanks to Lemma \ref{l.mult.loc.es.1}, it can be seen that the previous statement is true for $(z,V)\in Y_T(\Omega)$. More precisely, if we call
$$ Z=\{(z,V)\in Y_T(\Omega)\mid z(0)=u_0,\quad V(0)=Q_0\}, $$
then $\Phi\colon Z\to Z$ is well-defined and we have 
\begin{equation}\label{proof.loc.ex.1}
    \begin{aligned}
        \|(u,Q)\|_{Y_T(\Omega)}\le C(Q_0)\left[\|u_0\|_{H^1(\Omega)} + \|Q_0\|_{H^2(\Omega)} + T^\gamma \left(1+\|(z,V)\|_{Y_T(\Omega)}^5\right)\right] + \\
    +  C(Q_0)\|V-Q_0\|_{L^\infty_TH^{s+1}}\|(z,V)\|_{Y_T(\Omega)}\left(\|(z,V)\|_{Y_T(\Omega)} + \|Q_0\|_{H^2(\Omega)}\right),
    \end{aligned}
\end{equation}
for some $\gamma>0$ and $s\in(1/2,1)$ and where we recall
$$ C(Q_0)\sim 1+\|Q_0\|_{H^2(\Omega)}^2. $$
Let us define now 
$$ Z_{\omega,\delta}=\{(z,V)\in Z\mid \|(z,V)\|_{Y_T(\Omega)}\le \omega,\quad \|(z,V)-(u_0,Q_0)\|_{L^\infty_T(H^s\times H^{s+1})}\le \delta\}. $$
We want to prove that $\Phi\colon Z_{\omega,\delta}\to Z_{\omega,\delta}$. We already know from \eqref{proof.loc.ex.1} that there is $M>0$ such that  
$$ \|(u,Q)\|_{Y_T(\Omega)}\le M\left(1+\|Q_0\|_{H^2(\Omega)}^2\right)\left[\|u_0\|_{H^1(\Omega)} + \|Q_0\|_{H^2(\Omega)}\right] + $$
$$ + M\left(1+\|Q_0\|_{H^2(\Omega)}^2\right)\left[T^\gamma \left(1+\omega^5\right) + \delta\omega\left(\omega + \|Q_0\|_{H^2(\Omega)}\right)\right]. $$
So, if we take $\omega$ such that
$$ M\left(1+\|Q_0\|_{H^2(\Omega)}^2\right)\left(\|u_0\|_{H^1(\Omega)} + \|Q_0\|_{H^2(\Omega)}\right)=\frac{\omega}{3} $$
and we take $\delta,T>0$ sufficiently small such that
$$ M\left(1+\|Q_0\|_{H^2(\Omega)}^2\right)T^\gamma \left(1+\omega^5\right)\le \frac{\omega}{3}, $$
$$ M\delta\left(1+\|Q_0\|_{H^2(\Omega)}^2\right)\left(\omega + \|Q_0\|_{H^2(\Omega)}\right)\le \frac{1}{3}, $$
we obtain
$$ \|\Phi(z,V)\|_{Y_T(\Omega)}\le \omega. $$
Moreover, from Proposition \ref{p.ex.i-lin.sys.} and the estimates of Lemmas \ref{l.mult.loc.es.0} and \ref{l.mult.loc.es.1}, for $T>0$ sufficiently small it holds
$$ \|\Phi(z,V)(t)-(u_0,Q_0)\|_{H^s\times H^{s+1}(\Omega)}   $$
$$ \lesssim \frac{\delta}{2} + T^\frac{1-s}{2}C(Q_0)\|\widetilde f(V,Q_0) + \widetilde{F}(z,V)\|_{L^2_TL^2}  + T^\frac{1-s}{2}C(Q_0)\|\widetilde g(z,V,Q_0) +  \widetilde{G}(z,V)\|_{L^2_TH^1} $$
$$ \lesssim \frac{\delta}{2} + T^\frac{1-s}{2}C(Q_0)\left[T^\gamma \left(1+\omega^5\right) + \delta\omega\left(\omega + \|Q_0\|_{H^2(\Omega)}\right)\right]. $$
Since we chose $s\in(1/2,1)$, we can find $T>0$ sufficiently small such that
$$ \|\Phi(z,V)(t)-(u_0,Q_0)\|_{H^s\times H^{s+1}(\Omega)}\le \delta $$
and therefore $\Phi\colon Z_{\omega,\delta}\to Z_{\omega,\delta}$. Let us consider now $(z_1,V_1),(z_2,V_2)\in Z_{\omega,\delta}$. Clearly
$$ \| V_1-V_2\|_{L^\infty_TH^{s+1}}\le \| V_1-Q_0\|_{L^\infty_TH^{s+1}} + \| V_2-Q_0\|_{L^\infty_TH^{s+1}} \le 2\delta. $$
So, similarly to what we did previously, it holds
\begin{equation}\label{proof.nl.loc.es}
\begin{aligned}
    &\|\Phi(z_1,V_1)-\Phi(z_2,V_2)\|_{Y_T(\Omega)} \\
    \le C(Q_0) \left[T^\gamma \left(1+\omega^4\right)\right.&\left. + \delta\left(\omega + \|Q_0\|_{H^2(\Omega)}\right)\right]\|(z_1,V_1)-(z_2,V_2)\|_{Y_T(\Omega)}.
\end{aligned}
\end{equation}
Therefore, for $T,\delta>0$ sufficiently small, $\Phi\colon Z_{\omega,\delta}\to Z_{\omega,\delta}$ is a contraction and we get the existence of a local solution in $Y_T(\Omega)$. The existence of the pressure $p$ comes from Remark \ref{rem.press.}.

Finally, it remains to prove the uniqueness of the solution: let us suppose $(u_1,Q_1)$ and $(u_2,Q_2)$ be two solutions in $Y_T(\Omega)$ and let
$$ R=\max\{\|(u_1,Q_1)\|_{Y_T(\Omega)},\|(u_2,Q_2)\|_{Y_T(\Omega)}\}. $$
We want to apply the estimate \eqref{proof.nl.loc.es} but, in order to do so, we need to prove that 
$$ \|(u_j,Q_j)-(u_0,Q_0)\|_{L^\infty_T(H^s\times H^{s+1})}\le \delta \quad j=1,2. $$
As before, thanks to Proposition \ref{p.ex.i-lin.sys.}, we have for $j=1,2$ 
$$ \|(u_j,Q_j)-(u_0,Q_0)\|_{L^\infty_T(H^s\times H^{s+1})} $$
$$ \lesssim \frac{\delta}{2} + T^\frac{1-s}{2}C(Q_0)\left[T^\gamma \left(1+R^5\right) + \|(u_j,Q_j)-(u_0,Q_0)\|_{L^\infty_T(H^s\times H^{s+1})}R\left(R + \|Q_0\|_{H^2(\Omega)}\right)\right]. $$
So for $T_0=T_0(R,Q_0)>0$ sufficiently small
$$ \|(u_j,Q_j)-(u_0,Q_0)\|_{L^\infty_{T_0}(H^s\times H^{s+1})}\le \delta \quad j=1,2. $$
Finally, from \eqref{proof.nl.loc.es}
$$ \|(u_1,Q_1)-(u_2,Q_2)\|_{Y_{T_0}(\Omega)} $$
$$ \le C(Q_0) \left[T_0^\gamma \left(1+R^4\right) + \delta\left(R + \|Q_0\|_{H^2(\Omega)}\right)\right]\|(u_1,Q_1)-(u_2,Q_2)\|_{Y_{T_0}(\Omega)}. $$
So, for $T_0>0$ sufficiently small, we get that $(u_1,Q_1)=(u_2,Q_2)$ for a.e. $(t,x)\in (0,T_0)\times\Omega$. In order to conclude, it is sufficient to notice that we can repeat the argument with a different starting time: in fact, the choice of $T_0$ depends on $Q_0$ but only due to its presence in the nonlinear terms $f$ and $g$. Therefore, if we consider the same problem in $(t_0,+\infty)$ with $t_0=T_0-\varepsilon$, we can prove as before that $(u_1,Q_1)=(u_2,Q_2)$ for a.e. $(t,x)\in(0,2T_0-\varepsilon)\times\Omega$. Therefore, in a finite number of steps we conclude $(u_1,Q_1)=(u_2,Q_2)$ for a.e. $(t,x)\in(0,T)\times\Omega$.  
\end{proof}
As a corollary of the local existence result and the $L^p-L^q$ maximal regularity inequality from Theorem \ref{t.max.reg.}, we have Corollary \ref{c.loc.ex.pq.}:
\begin{proof}[Proof of Corollary \ref{c.loc.ex.pq.}]\hfill\\
Thanks to Theorem \ref{t.max.reg.} and the local well-posedness, it is sufficient to prove that 
$$ \|F(u,Q)\|_{L^p_TL^q} + \|G(u,Q)\|_{L^p_TW^{1,q}}\lesssim C(\|(u,Q)\|_{Y_T(\Omega)})<\infty.  $$
In order to do so, we need to repeat the inequalities of Lemmas \ref{l.mult.loc.es.0} and \ref{l.mult.loc.es.1} in the $L^p_TL^q$ setting. We show the details only of
$$ \|\nabla v w\|_{L^p_TW^{1,q}}\le C(\|v\|_{X^1_T(\Omega)}, \|w\|_{X^2_T(\Omega)}), \quad k\in\mathbb N, $$
for $v\in X^1_T(\Omega)$ and $w\in X^2_T(\Omega)$. Let us start from the $L^p_TL^q$-norm: let $r>2$ such that 
$$ r=\left\{\begin{array}{ll}
    \frac{2q}{2-q} & q\in(1,2) \\
    \infty & q=2.
\end{array}\right. $$
Then by H\"older inequality
$$ \|\nabla vw\|_{L^p_TL^q}\le \left\|\|\nabla v\|_{L^2(\Omega)}\|w\|_{L^r(\Omega)}\right\|_{L^p((0,T))} $$
$$ \lesssim \|\nabla v\|_{L^\infty_TL^2}\|w\|_{L^p_TH^2}\le T^\frac{2-p}{2p}\|v\|_{X^1_T(\Omega)}\|w\|_{X^2_T(\Omega)}, $$
which is bounded, since $p\le 2$. The estimate of 
$$ \|\nabla(\nabla v w)\|_{L^p_TL^q}\le \|\nabla^2v\cdot w\|_{L^p_TL^q} + \|\nabla v \cdot \nabla w\|_{L^p_TL^q}. $$
is similar to the previous one. Therefore
$$ \|f(Q)\|_{L^p_TL^q} + \|g(u,Q)\|_{L^p_TW^{1,q}}\lesssim C(\|(u,Q)\|_{Y_T(\Omega)})  $$
and by Theorem \ref{t.max.reg.} we conclude.
\end{proof}

\section{Global well-posedness}\label{sec.global}
\subsection{Linear estimates}

Let us focus now on global in time solutions. As mentioned in the introduction, the global well-posedness result is restricted to the case $N=3$. For the global well-posedness, it is sufficient to prove an a priori estimate on the linear system to proceed with the contraction argument. In particular, we firstly consider
\begin{equation}\label{BE.lin.sys.gl.}
    \left\{\begin{array}{ll}
        (\partial_t-\Delta)u+\nabla \pi+\beta {\rm Div}(\Delta-a)Q= F  & \mathbb R_+\times \Omega \\
       (\partial_t-\Delta+a)Q-\beta D(u)=G  & \mathbb R_+\times \Omega \\
       {\rm div} u=0 & \mathbb R_+\times \Omega \\
       u=0,\quad \partial_\nu Q=0 & \mathbb R_+\times\partial\Omega \\
       u(0)=u_0,\quad Q(0)=Q_0 & \Omega,
\end{array}\right. 
\end{equation}
The technique we use is again the energy method. As we have already seen in the proof of Proposition \ref{p.res-unif.es.}, the estimate of the $L^2L^2$-norm of $u$ has to be done separately, with a different approach. For this reason, we first consider the auxiliary space for $u$ and $Q$
\begin{equation}\label{def.Xtilde}
 \widetilde X^1(\Omega)=\left\{u\in L^\infty\left(\mathbb R_+;H^1\left(\Omega;\mathbb R^3\right)\right)\:\Big|\:\nabla u\in L^2\left(\mathbb R_+;H^1\left(\Omega;\mathbb R^{3\times 3}\right)\right)\right\}   
\end{equation}
\begin{equation}\label{def.Ytilde}
    \widetilde Y(\Omega)=\widetilde X^1(\Omega)\times X^2(\Omega),
\end{equation}
where $X^2$ is defined in \eqref{def.X} with $T=\infty$. We note that $\widetilde Y(\Omega)$ is a Banach space endowed with the norm
$$ \|(u,Q)\|_{\widetilde Y(\Omega)}=\|u\|_{L^\infty H^1} + \|\nabla u\|_{L^2 H^1} + \|Q\|_{X^2(\Omega)}. $$
\begin{prop}\label{p.lin.es.gl.}
Let $\Omega=\mathbb R^3,\mathbb R^3_+$, let $a>0$ and $\beta\in\mathbb R$, let 
$$ u_0\in H^1_{\mathbb P\Delta_D}\left(\Omega;\mathbb R^3\right), \quad Q_0\in H^2_{\Delta_N}\left(\Omega;S_0(3,\mathbb R)\right), $$
$$ F\in L^2\left(\mathbb R_+; L^2\left(\Omega;\mathbb R^3\right)\right)\cap L^1\left(\mathbb R_+; L^2\left(\Omega;\mathbb R^3\right)\right), $$
$$ G\in L^2\left(\mathbb R_+;H^1\left(\Omega;S_0(3,\mathbb R)\right)\right)\cap L^1\left(\mathbb R_+;L^2\left(\Omega;S_0(3,\mathbb R)\right)\right), $$
then the system \eqref{BE.lin.sys.gl.} admits a solution $(u,\pi,Q)$, unique up to additive functions $c(t)$ on the pressure term, with $\pi(t)\in L^1_{loc}(\Omega)$ for a.e. $t>0$ such that 
$$ (u,Q)\in \widetilde Y(\Omega), \quad \nabla \pi\in L^2\left(\mathbb R_+;L^2\left(\Omega;\mathbb R^3\right)\right), $$
which satisfies 
$$ \|(u,Q)\|_{\widetilde Y(\Omega)} + \|\nabla \pi\|_{L^2(\mathbb R_+;L^2(\Omega))} $$
$$ \lesssim \|u_0\|_{H^1(\Omega)} + \|Q_0\|_{H^2(\Omega)} + \|(F,G)\|_{L^2(\mathbb R_+;L^2(\Omega)\times H^1(\Omega))} + \|(F,G)\|_{L^1(\mathbb R_+;L^2(\Omega)\times L^2(\Omega))}, $$
where we recall the definition of $\widetilde Y(\Omega)$ and its norm from \eqref{def.Ytilde}.
\end{prop}
\begin{proof}\hfill\\
Similarly to Proposition \ref{p.ex.i-lin.sys.}, we can prove that for $T>0$ sufficiently small there is a solution for the system \eqref{BE.lin.sys.gl.}, with 
$$ (u,Q)\in Y_T(\Omega), \quad \nabla \pi\in L^2((0,T);L^2(\Omega;\mathbb R^N)) $$
that satisfies the estimate
$$ \|(u,Q)\|_{Y_T(\Omega)} + \|\nabla \pi\|_{L^2_TL^2}\le C(T)\left[\|u_0\|_{H^1(\Omega)} + \|Q_0\|_{H^2(\Omega)} + \|f\|_{L^2_TL^2} + \|g\|_{L^2_TH^1}\right].  $$
In particular, $C(T)\lesssim 1$ if $T\le 1$. Our aim is to find an estimate for $(u,Q)$ that does not depend on $T$. In this way, we will be able to extend our solution for any $t>0$ in the proper norm. Since $(u,Q)$ is a solution in $(0,T)$, we can apply the energy method as before. We recall from Lemma \ref{l.res.limit} that 
$$ R_\varepsilon^Su\xrightarrow{\varepsilon\to 0^+}u\quad \text{in}\quad H^2(\Omega) $$
$$ R_\varepsilon^HQ\xrightarrow{\varepsilon\to0^+}Q\quad \text{in}\quad H^3(\Omega). $$
So, up to a density argument, we apply Lemma \ref{l.en-can.1}, Lemma \ref{l.en.met-LHS1} and Lemma \ref{l.en.met-LHS2} to obtain
\begin{equation}\label{proof.lin.gl.es.1}
    \|u\|_{L^\infty_TL^2}^2 + \|\nabla u\|_{L^2_TL^2}^2 +  \|Q\|_{X^1_T(\Omega)}^2 \lesssim \|u_0\|_{L^2(\Omega)}^2 + \|Q_0\|_{H^1(\Omega)}^2 + \|(F,G)\|_{L^1_TL^2}^2,
\end{equation}
\begin{equation}\label{proof.lin.gl.es.2}
    \begin{aligned}
        \|\nabla^\prime u&\|_{L^\infty_TL^2}^2 + \|\nabla \nabla^\prime u\|_{L^2_TL^2}^2 +  \|\nabla^\prime Q\|_{L^\infty_TH^1}^2 + \|\nabla^\prime Q\|_{L^2_TH^2}^2 \\
        &\lesssim \|u_0\|_{H^1(\Omega)}^2 + \|Q_0\|_{H^2(\Omega)}^2 + \|(F,\nabla G)\|_{L^2_TL^2}^2,
    \end{aligned}
\end{equation}
\begin{equation}\label{proof.lin.gl.es.3}
\begin{aligned}
     \|\partial_3u&\|_{L^\infty_TL^2}^2 + \|\nabla \partial_3u\|_{L^2_TL^2}^2 +  \|\partial_3Q\|_{L^\infty_TH^1}^2 + \|\partial_3Q\|_{L^2_TH^2}^2\\ 
    & \lesssim \|u_0\|_{H^1(\Omega)}^2 + \|Q_0\|_{H^2(\Omega)}^2 +  \|(F,\nabla G)\|_{L^2_TL^2}^2 + \\
     + &\|\nabla^\prime\nabla u\|_{L^2_TL^2}\|Q\|_{X^2_T(\Omega)} + \|\nabla^\prime Q\|_{L^2_TH^2}\|\nabla u\|_{X^0_T(\Omega)},
\end{aligned}
\end{equation}
where we used that
$$ \left|\int_0^t\int_\Omega F\cdot u dxd\tau\right|\le\|F\|_{L^1_TL^2}\|u\|_{L^\infty_TL^2}, $$
$$ \left|\int_0^t\int_\Omega F\cdot \partial_i^2u dxd\tau\right|\le \|F\|_{L^2_TL^2}\|\partial_i^2u\|_{L^2_TL^2}\quad i=1,2,3 $$
$$ \left|\int_0^t\int_\Omega G\colon (a-\Delta)Q dxd\tau\right|\le \|G\|_{L^1_TL^2}\|Q\|_{L^\infty_TH^2}, $$
$$ \left|\int_0^t\int_\Omega \partial_iG\colon \partial_i(a-\Delta)Q dxd\tau\right|\le \|\partial_i G\|_{L^2_TL^2}\|\partial_iQ\|_{L^2_TH^2}\quad i=1,2,3. $$
To be noticed that the constants in the right hand side of the inequalities \eqref{proof.lin.gl.es.1}-\eqref{proof.lin.gl.es.3} does not depend on $T$. In particular,  combining \eqref{proof.lin.gl.es.1}, \eqref{proof.lin.gl.es.2} and \eqref{proof.lin.gl.es.3}, we get
\begin{equation}\label{proof.lin.gl.es.5}
\begin{aligned}
    &\|u\|_{L^\infty_TH^1}^2 + \|\nabla u\|_{L^2_TH^1}^2 +  \|Q\|_{L^\infty_TH^2}^2 + \|Q\|_{L^2_TH^3}^2\\
    \le C&\left[\|u_0\|_{H^1(\Omega)}^2 + \|Q_0\|_{H^2(\Omega)}^2 +  \|(F,\nabla G)\|_{L^2_TL^2}^2 +  \|(F,G)\|_{L^1_TL^2}^2\right],
\end{aligned}
\end{equation}
for some $C>0$ independent of $T$. Therefore, since 
$$ F\in L^1(\mathbb R_+;L^2(\Omega;\mathbb R^N))\cap L^2(\mathbb R_+;L^2(\Omega;\mathbb R^N)), $$
$$ G\in L^1(\mathbb R_+;L^2(\Omega;S_0(N,\mathbb R)))\cap L^2(\mathbb R_+;H^1(\Omega;S_0(N,\mathbb R))), $$
the solution $(u,\pi,Q)$ can be extended to $\mathbb R_+$ and taking the limit as $T\to+\infty$ on the inequality \eqref{proof.lin.gl.es.5} we conclude. 
\end{proof}
As anticipated, we now focus on the $L^2L^2$-norm of $u$. To control this quantity, it becomes fundamental the decay in time of the semigroup corresponding to the linear $Q$-tensor system. When $\Omega=\mathbb R^3$ the decay of the semigroup has already been studied in \cite{MS22} for inhomogeneous Sobolev spaces (see Theorem 3.1 of \cite{MS22}) 
\begin{lem}\label{l.sem.dec.es.RN}
Let $N\ge 3$, let $p,q\in(1,\infty)$ with $1<p<2\le q<\infty$, let $B$ as defined in \eqref{def.op.B} and 
$$ (u_0,Q_0)\in L^p\left(\mathbb R^N;\mathbb R^N\right)\times W^{1,p}\left(\mathbb R^N;S_0(N,\mathbb R)\right), $$
then 
$$ \left\|\nabla^j e^{B t}(u_0,Q_0)\right\|_{L^q(\mathbb R^N)\times W^{1,q}(\mathbb R^N)} $$
$$ \lesssim t^{-\frac{j}{2}-\frac{N}{2}\left(\frac{1}{p}-\frac{1}{q}\right)}\left[\|(u_0,Q_0)\|_{L^p(\mathbb R^N)\times W^{1,p}(\mathbb R^N)} + \|(u_0,Q_0)\|_{L^q(\mathbb R^N)\times W^{1,q}(\mathbb R^N)}\right]\quad j=0,1,2. $$
\end{lem}
When $\Omega=\mathbb R^3_+$, we can imply the decay estimate from the work \cite{BMS25}:
\begin{lem}\label{l.sem.dec.es.hs}
Let $N\ge 2$, let $1<p<q<\infty$, let $B$ as defined in \eqref{def.op.B} and 
$$ (u_0,Q_0)\in L^p\left(\mathbb R^N_+;\mathbb R^N\right)\times \dot H^1_p\left(\mathbb R^N_+;S_0(N,\mathbb R)\right), $$
then 
$$ \left\|\nabla^j e^{B t}(u_0,Q_0)\right\|_{L^q(\mathbb R^N_+)\times \dot H^1_q(\mathbb R^N_+)}\lesssim t^{-\frac{j}{2}-\frac{N}{2}\left(\frac{1}{p}-\frac{1}{q}\right)}\|(u_0,Q_0)\|_{L^p(\mathbb R^N_+)\times \dot H^1_p(\mathbb R^N_+)}\quad j=0,1. $$
\end{lem}
\begin{proof}\hfill\\
Let $v_0=(u_0,Q_0)$. Lemma 4.3 in \cite{BMS25} yields
$$ \|e^{Bt}v_0\|_{L^{r_2}(\mathbb R^N_+)\times \dot H^1_{r_2}(\mathbb R^N_+)}
\lesssim
t^{-\frac N2\left(\frac1{r_1}-\frac1{r_2}\right)}
\|v_0\|_{L^{r_1}(\mathbb R^N_+)\times \dot H^1_{r_1}(\mathbb R^N_+)}, $$
whenever $\frac N2\left(\frac1{r_1}-\frac1{r_2}\right)\le 1$. With the same strategy, it can be proved that
\begin{equation}\label{eq:decay-step}
\|\nabla^j e^{Bt}v_0\|_{L^{r_2}(\mathbb R^N_+)\times \dot H^1_{r_2}(\mathbb R^N_+)}
\lesssim
t^{-\frac j2-\frac N2\left(\frac1{r_1}-\frac1{r_2}\right)}
\|v_0\|_{L^{r_1}(\mathbb R^N_+)\times \dot H^1_{r_1}(\mathbb R^N_+)},
\end{equation}
for $\frac N2\left(\frac1{r_1}-\frac1{r_2}\right)\le 1$ and $j=0,1$. Let now $1<p<q<\infty$ be arbitrary and let $m\in\mathbb N$ such that
$$ \frac N2\Bigl(\frac1p-\frac1q\Bigr)\le m. $$
We define intermediate exponents $\{r_k\}_{k=0}^m$ by
\[
r_0=p,\qquad r_m=q,\qquad 
\frac1{r_k}=\frac{m-k}{m}\frac1p+\frac{k}{m}\frac1q\quad \text{for}\:\: k=0,\dots,m.
\]
Then for each $k=1,\dots,m$,
\[
\frac N2\Bigl(\frac1{r_{k-1}}-\frac1{r_k}\Bigr)
=\frac1m\,\frac N2\Bigl(\frac1p-\frac1q\Bigr)\le 1,
\]
so that \eqref{eq:decay-step} applies to every consecutive pair $(r_{k-1},r_k)$. Using the semigroup property $e^{Bt}=(e^{B t/m})^m$ and iterating \eqref{eq:decay-step} $m$ times, we obtain for $j=0,1$,
\begin{align*}
\|\nabla^j e^{Bt}(v_0)\|_{L^{q}(\mathbb R^N_+)\times \dot H^1_{q}(\mathbb R^N_+)}
&=
\Bigl\|\nabla^j \bigl(e^{B t/m}\bigr)^m (v_0)\Bigr\|_{L^{r_m}(\mathbb R^N_+)\times \dot H^1_{r_m}(\mathbb R^N_+)}\\
&\lesssim 
\left(\frac{t}{m}\right)^{-\frac j2-\sum_{k=1}^m\frac N2\left(\frac1{r_{k-1}}-\frac1{r_k}\right)}
\|v_0\|_{L^{r_0}(\mathbb R^N_+)\times \dot H^1_{r_0}(\mathbb R^N_+)}\\
&=
\left(\frac{t}{m}\right)^{-\frac j2-\frac N2\left(\frac1p-\frac1q\right)}
\|v_0\|_{L^{p}(\mathbb R^N_+)\times \dot H^1_{p}(\mathbb R^N_+)}.
\end{align*}
Absorbing the factor $m^{\frac j2+\frac N2(\frac1p-\frac1q)}$ into the implicit constant we conclude.
\end{proof}
We are finally ready to state the full linear estimate for the system \eqref{BE.lin.sys.gl.}:
\begin{prop}\label{p.lin.es.gl.2}
Let $\Omega=\mathbb R^3,\mathbb R^3_+$, let $a>0$ and $\beta\in\mathbb R$, let $q\in \left(1,\frac{6}{5}\right)$ and 
$$ u_0\in H^1_{\mathbb P\Delta_D}\left(\Omega;\mathbb R^3\right)\cap L^q\left(\Omega;\mathbb R^3\right), $$
$$ Q_0\in H^2_{\Delta_N}\left(\Omega;S_0(3,\mathbb R)\right)\cap W^{1,q}\left(\Omega;S_0(3,\mathbb R)\right), $$
$$ F\in L^2\left(\mathbb R_+;L^2\left(\Omega;\mathbb R^3\right)\right)\cap L^1\left(\mathbb R_+;L^2\left(\Omega;\mathbb R^3\right)\cap L^q\left(\Omega;\mathbb R^3\right)\right), $$
$$ G\in L^2\left(\mathbb R_+;H^1\left(\Omega;S_0(3,\mathbb R)\right)\right)\cap L^1\left(\mathbb R_+;H^1\left(\Omega;S_0(3,\mathbb R)\right)\cap W^{1,q}\left(\Omega;S_0(3,\mathbb R)\right)\right), $$
then the system \eqref{BE.lin.sys.gl.} admits a solution $(u,\pi,Q)$, unique up to additive functions $c(t)$ on the pressure term, with $\pi(t)\in L^1_{loc}(\Omega)$ for a.e. $t>0$ such that 
$$ (u,Q)\in  Y(\Omega), \quad \nabla \pi\in L^2\left(\mathbb R_+;L^2\left(\Omega;\mathbb R^3\right)\right), $$
that satisfies 
$$ \|(u,Q)\|_{Y(\Omega)} + \|\nabla \pi\|_{L^2(\mathbb R_+;L^2(\Omega))}\lesssim \|u_0\|_{L^2(\Omega)\cap L^q(\Omega)} + \|Q_0\|_{H^1(\Omega)\cap W^{1,q}(\Omega)} +  $$
$$ +  \|(F,G)\|_{L^2(\mathbb R_+;L^2(\Omega)\times H^1(\Omega))} + \|(F,G)\|_{L^1(\mathbb R_+;L^2(\Omega)\times H^1(\Omega))}  + \|(F,G)\|_{L^1(\mathbb R_+;L^q(\Omega)\times W^{1,q}(\Omega))}. $$
\end{prop}
\begin{proof}\hfill\\
Let $(u,\pi,Q)$ from Proposition \ref{p.lin.es.gl.}. Since $(u,Q)$ solves the system \eqref{BE.lin.sys.gl.}, it satisfies the Duhamel Formula
$$ (u,Q)(t)=e^{Bt}(u_0,Q_0) + \int_0^t e^{B(t-\tau)}(F,G)(\tau)d\tau. $$
In particular
\begin{equation}\label{Duhamel.aux.}
    u(t)=\left[e^{Bt}(u_0,Q_0)\right]_1 + \int_0^t\left[e^{B(t-\tau)}(F,G)(\tau)\right]_1d\tau,
\end{equation}
when $[\cdot]_j$ denote the $j$-th component for $j=1,2$. As in the proof of Proposition \ref{p.res-unif.es.}, we can control the $L^2L^2$-norm in the time interval $(0,1)$:
$$ \|u\|_{L^2_1L^2}\le \|(u_0,Q_0)\|_{L^2(\Omega)\times H^1(\Omega)} + \|F\|_{L^1L^2} + \|G\|_{L^1H^1}. $$
We can then focus on the case $t\ge 1$. Thanks to Lemma \ref{l.sem.dec.es.RN} and \ref{l.sem.dec.es.hs}, we know that 
$$ \left\|[e^{Bt}(u_0,Q_0)]_1\right\|_{L^2(\Omega)}\lesssim t^{-\frac{3}{2}\left(\frac{1}{q}-\frac{1}{2}\right)}\left[\|u_0\|_{L^2(\Omega)\cap L^q(\Omega)} + \|Q_0\|_{H^1(\Omega)\cap W^{1,q}(\Omega)}\right]. $$
So, since $q<\frac{6}{5}$, the first term of \eqref{Duhamel.aux.} belongs to $L^2L^2$. Concerning the inhomogeneous term
$$ \int_0^t\left[e^{B(t-\tau)}(F,G)(\tau)\right]_1d\tau= \int_0^{t-1}\left[e^{B(t-\tau)}(F,G)(\tau)\right]_1d\tau + \int_{t-1}^t\left[e^{B(t-\tau)}(F,G)(\tau)\right]_1d\tau. $$
Thanks to Lemma \ref{l.sem.dec.es.RN} and \ref{l.sem.dec.es.hs}
$$ \left\|\int_0^{t-1}\left[e^{B(t-\tau)}(F,G)(\tau)\right]_1d\tau \right\|_{L^2(\Omega)} $$
$$ \lesssim \int_0^{t-1}(t-\tau)^{-\frac{3}{2}\left(\frac{1}{q}-\frac{1}{2}\right)}\left[\|F(\tau)\|_{L^2(\Omega)\cap L^q(\Omega)} + \|G(\tau)\|_{H^1(\Omega)\cap  W^{1,q}(\Omega)}\right]d\tau $$
$$ \le \int_{\mathbb R}(t-\tau)^{-\frac{3}{2}\left(\frac{1}{q}-\frac{1}{2}\right)}1_{(1,\infty)}(t-\tau)\left[\|F(\tau)\|_{L^2(\Omega)\cap L^q(\Omega)} + \|G(\tau)\|_{H^1(\Omega)\cap  W^{1,q}(\Omega)}\right]d\tau, $$
while, applying Theorem \ref{t.sem.es.} 
$$ \left\|\int_{t-1}^t\left[e^{B(t-\tau)}(F,G)(\tau)\right]_1d\tau \right\|_{L^2(\Omega)}\lesssim \int_{t-1}^t\|(F,G)(\tau)\|_{L^2(\Omega)\times H^1(\Omega)}d\tau $$
$$ = \int_{\mathbb R}1_{(0,1)}(t-\tau)\|(F,G)(\tau)\|_{L^2(\Omega)\times H^1(\Omega)}d\tau. $$
Therefore, by Young's inequality 
$$ \left\|\int_0^{t-1}\left[e^{B(t-\tau)}(F,G)(\tau)\right]_1d\tau \right\|_{L^2((1,\infty);L^2(\Omega))} $$
$$ \lesssim \|t^{-\frac{3}{2}\left(\frac{1}{q}-\frac{1}{2}\right)}\|_{L^2((1,\infty))}\left[\|F\|_{L^1(L^2\cap L^q)} + \|G\|_{L^1(H^1\cap W^{1,q})}\right], $$
$$ \left\|\int_{t-1}^t\left[e^{B(t-\tau)}(F,G)(\tau)\right]_1d\tau \right\|_{L^2((1,\infty);L^2(\Omega))}\lesssim \|(F,G)\|_{L^1(L^2\times H^1)}. $$
Again, thanks to the choice of $q$ we conclude. 
\end{proof}

\subsection{Proof of Theorem \ref{t.gl.ex.2}}

To apply Proposition \ref{p.lin.es.gl.2}, we need to verify that 
$$ F(u,Q)\in L^2L^2\cap L^1(L^2\cap L^q), $$
$$ G(u,Q)\in L^2H^1\cap L^1(H^1\cap W^{1,q}), $$
for some $q\in\left(1,\frac{6}{5}\right)$. 
\begin{lem}\label{l.nonlin.es.gl.}
Let $\Omega=\mathbb R^3,\mathbb R^3_+$,let $k\in\mathbb N$, let $v_1,v_2\in X^1(\Omega)$ and $w_j\in X^2(\Omega)$ for $j=1,\ldots,k$, then it holds 
$$ \|v_1\nabla v_2\|_{L^2(\mathbb R_+;L^2(\Omega))}\lesssim \|v_1\|_{X^1(\Omega)}\|v_2\|_{ X^1(\Omega)}, $$
$$ \|v_1\nabla v_2 w_1\cdots w_k\|_{L^2(\mathbb R_+;L^2(\Omega))}\lesssim \|v_1\|_{X^1(\Omega)}\|v_2\|_{X^1(\Omega)}\prod_{j=1}^k \|w_j\|_{X^2(\Omega)}, $$
$$ \|v_1w_1\cdots w_k\|_{L^2(\mathbb R_+;L^2(\Omega))}\lesssim \|v_1\|_{X^1(\Omega)}\prod_{j=1}^k\|w_j\|_{X^2(\Omega)}, $$
$$ \|w_1\cdots w_k\|_{L^2(\mathbb R_+;L^2(\Omega))}\lesssim \prod_{j=1}^k\|w_j\|_{X^2(\Omega)}. $$
\end{lem}
\begin{lem}\label{l.nonlin.gl.es.2}
Let $\Omega=\mathbb R^3,\mathbb R^3_+$, let $k\in\mathbb N$ and $q\in(1,2]$, let $v_1,v_2\in X^1(\Omega)$ and $w_j\in X^2(\Omega)$ for $j=1,\ldots, k$, then it holds 
$$ \|\nabla v_1w_1\cdots w_k\|_{L^1(\mathbb R_+;W^{1,q}(\Omega))}\lesssim \|v_1\|_{X^1(\Omega)} \prod_{j=1}^k\|w_j\|_{X^2(\Omega)}, $$
$$ \|v_1\nabla v_2\|_{L^1(\mathbb R_+;L^q(\Omega))}\lesssim \|v_1\|_{X^1(\Omega)}\|v_2\|_{X^1(\Omega)}, $$
$$ \|v_1\nabla v_2 w_1\cdots w_k\|_{L^1(\mathbb R_+;L^q(\Omega))}\lesssim \|v_1\|_{X^1(\Omega)}\|v_2\|_{X^1(\Omega)}\prod_{j=1}^k \|w_j\|_{X^2(\Omega)}, $$
$$ \|v_1w_1\cdots w_k\|_{L^1(\mathbb R_+;L^q(\Omega))}\lesssim \|v_1\|_{X^1(\Omega)}\prod_{j=1}^k\|w_j\|_{X^2(\Omega)}, $$
$$ \|w_1\cdots w_k\|_{L^1(\mathbb R_+;L^q(\Omega))}\lesssim \prod_{j=1}^k\|w_j\|_{X^2(\Omega)}  \quad k\ge2. $$
\end{lem}
We omit the proofs, which are based on Sobolev and H\"older inequalities. We are finally ready to prove Theorem \ref{t.gl.ex.2}:
\begin{proof}[Proof of Theorem \ref{t.gl.ex.2}]\hfill\\
We consider the map $\Phi\colon (z,V)\mapsto (u,Q)$, where $(u,Q)$ solves the system
$$ \left\{\begin{array}{ll}
    (\partial_t-\mathbb P\Delta)u +\beta\mathbb P{\rm Div}(\Delta-a)Q = \mathbb P F(z,V)  & \mathbb R_+\times\Omega \\
    {\rm div}u=0 & \mathbb R_+\times\Omega \\
    (\partial_t+a-\Delta)Q - \beta D(u) = G(z,V) & \mathbb R_+\times \Omega \\
    u=0, \quad \partial_\nu Q=0 & \mathbb R_+\times\partial\Omega \\
    u(0)=u_0,\quad Q(0)=Q_0 & \Omega.
\end{array}\right. $$
Thanks to Proposition \ref{p.lin.es.gl.2}, such $(u,Q)$ exists in $Y(\Omega)$ if
$$ F(z,V)\in L^2(\mathbb R_+;L^2(\Omega;\mathbb R^3))\cap L^1(\mathbb R_+;L^2(\Omega;\mathbb R^3)\cap L^q(\Omega;\mathbb R^3)), $$
$$ G(z,V)\in L^2(\mathbb R_+;H^1(\Omega;S_0(3,\mathbb R)))\cap L^1(\mathbb R_+;H^1(\Omega;S_0(3,\mathbb R)))\cap W^{1,q}(\Omega;S_0(3,\mathbb R))), $$
which follows from Lemma \ref{l.mult.loc.es.0}, Lemma  \ref{l.nonlin.es.gl.} and Lemma \ref{l.nonlin.gl.es.2}. More precisely, if we call
$$ Z=\{(z,V)\in Y(\Omega)\mid z(0)=u_0,\quad V(0)=Q_0\}, $$
then $\Phi\colon Z\to Z$ is well-defined and we have 
\begin{equation}\label{proof.gl.ex.1}
        \|(u,Q)\|_{Y(\Omega)}\le C\left[\|u_0\|_{L^q(\Omega)\cap H^1(\Omega)} + \|Q_0\|_{W^{1,q}(\Omega)\cap H^2(\Omega)} + \|(z,V)\|_{Y(\Omega)}^2\left(1+\|(z,V)\|_{ Y(\Omega)}^3\right)\right],
\end{equation}
for some $C>0$. Let us define now 
$$ Z_{\varepsilon}=\{(z,V)\in Z\mid \|(z,V)\|_{Y(\Omega)}\le 2C\varepsilon\}. $$
We want to prove that $\Phi\colon Z_{\varepsilon}\to Z_{\varepsilon}$. If $(z,V)\in Z_\varepsilon$, then from \eqref{proof.gl.ex.1} it holds
$$ \|(u,Q)\|_{Y(\Omega)}\le C\varepsilon + 4C^2\varepsilon^2(1+8C^3\varepsilon^3). $$
So, if we take $\varepsilon$ sufficiently small, we get that $\Phi(z,V)\in Z_\varepsilon$. Similarly, if we consider $(z_1,V_1),(z_2,V_2)\in Z_{\varepsilon}$, we get that
\begin{equation}\label{proof.nl.gl.es}
    \|\Phi(z_1,V_1)-\Phi(z_2,V_2)\|_{ Y(\Omega)}\lesssim \varepsilon\left(1 + \varepsilon^3 \right)\|(z_1,V_1)-(z_2,V_2)\|_{\widetilde Y(\Omega)}.
\end{equation}
Therefore, for $\varepsilon$ sufficiently small, $\Phi\colon Z_{\varepsilon}\to Z_{\varepsilon}$ is a contraction and we get the existence of a global solution. Finally, the uniqueness can be proved as in the local case.
\end{proof}

\vspace{1cm}

\noindent \textbf{Conflict of Interest}: On behalf of all authors, the corresponding author states that there is no conflict of interest.

\noindent \textbf{Data Availability}: The manuscript has no associated data.

\clearpage
\bibliographystyle{plain}
 \bibliography{EL_BG}

\noindent {\bf D. Barbera}:
 \begin{itemize}
     \item [Affiliation:] Department of Mathematical Sciences "Giuseppe Luigi Lagrange", Politecnico di Torino, Corso Duca degli Abruzzi 24, 10129 Torino, Italy
     \item [Email:]  daniele.barbera96@gmail.com
 \end{itemize}
{\bf V. Georgiev}: 
  \begin{itemize}
     \item [Affiliation:] Department of Mathematics, University of Pisa, Largo Bruno Pontecorvo 5, 56100 Pisa, Italy
     \item [-] Faculty of Science and Engineering, Waseda University, 3-4-1, Okubo, Shinjuku-ku, Tokyo 169-8555, Japan
     \item [-] Institute of Mathematics and Informatics, Bulgarian Academy of Sciences, Acad. Georgi Bonchev Str., Block 8, 1113 Sofia, Bulgaria
    \end{itemize}
{\bf M. Murata}: 
  \begin{itemize}
     \item [Affiliation:] Department of Mathematical and System Engineering, Faculty of Engineering, Shizuoka University, 3-5-1 Johoku, Chuo-ku, Hamamatsu-shi, Shizuoka, 432-8561, Japan
 \end{itemize}
 {\bf Y. Shibata}: 
  \begin{itemize}
     \item [Affiliation:] Emeritus Professor of Waseda University, Waseda University, 3-4-1 Ohkubo Shinjuku-ku Tokyo, 169-8555, Japan
     \item [-] Adjunct faculty member in the Department of Mechanical Engineering and Materials Science, University of Pittsburgh, United States of America
 \end{itemize}

\end{document}